\documentclass[12pt,letterpaper]{amsart}
\usepackage{sty_alg_G_degree}

\begin{document}
\title{Local multiplicities for an equivariantly enriched non-transverse B\'ezout's Theorem}

\author[Bethea]{Candace Bethea}
\author[Ravi]{Charanya Ravi}

\begin{abstract}
We introduce the degree and local degree in equivariant motivic homotopy theory for the purpose of studying equivariant enumerative problems over general fields. Given a finite, tame group scheme $G$ over a field $k$ and an equivariant motivic ring spectrum $E_G$, we define the equivariant motivic degree and a corresponding local degree of a relatively $E_G$-oriented, proper, quasi-smooth morphism of $G$-schemes. We prove a local to global formula expressing the global degree as a sum of local contributions over $G$-orbits.

Using these constructions, we define the Euler number of an oriented vector bundle on a quasi-smooth, proper derived stack and show that the Euler number is independent of the choice of section under appropriate hypotheses. In the presence of a finite group action, the equivariant Euler number can be computed as a sum of local equivariant degrees. As an application, we obtain an equivariantly enriched local multiplicity formula for an equivariant non-transverse B\'ezout theorem, expressing an equivariant intersection number as a sum of local equivariant degrees.
\end{abstract}
\maketitle 

\begingroup
\setlength{\parskip}{3pt}
\renewcommand{\baselinestretch}{0.25}\normalsize
\tableofcontents
\renewcommand{\baselinestretch}{1.0}\normalsize
\endgroup

\section{Introduction}\label{section:Introduction}

Enriched enumerative geometry has seen a flurry of activity over recent decades, spurred by advancements in real, quadratically enriched, equivariantly enriched, and random enumerative geometry. See \cite{bb-history} for a historical overview. Quadratically enriched and, more recently, equivariantly enriched, enumerative geometry are of particular interest to homotopy theorists working in enumerative geometry, leveraging invariants in motivic and equivariant homotopy theory respectively to enrich classical enumerative results. The triumph of quadratically enriched enumerative geometry is the ability to count solutions to enumerative problems over non-closed fields, see for example \cite{KW-local-degree, 27lines, Hoyois-lefschetz, SW-4-lines, larson-vogt-bitangents, quadratic-tropical, quadratic-twisted-cubics,KLSW-rational-curves}. 
Likewise, one goal of equivariantly enriched enumerative geometry is to use tools from equivariant homotopy theory to count orbits of solutions to enumerative problems given a finite group action on the ambient space in question, see for example \cite{betheapencil, brazeuler, BB-bitangents, BW-G-degree}.

A reasonable starting point for both enriched enumerative theories is constructing the degree and local degree. In motivic homotopy theory, the degree of an endomorphism of $\mathbb{P}^1_k$ was defined by Morel as an element of the Grothendieck--Witt ring $\mathrm{GW}(k)$ \cite[Corollary 4.11]{morel-icm}. A more recent treatment of the quadratically enriched degree is that of Kass--Levine--Solomon--Wickelgren, \cite[2.2]{KLSW-rational-curves}, who define the quadratically enriched degree of a relatively oriented, finite, flat, local complete intersection morphism  $f\colon X\to Y$ of schemes over a field $k$,  as a non-degenerate, symmetric bilinear form on $\calO_Y$, valued in the Grothendieck--Witt sheaf of $Y$, $\mathcal{GW}(Y)$ \cite[2.2]{KLSW-rational-curves}. This generalizes previous work of Kass--Wickelgren, who define the local degree of a polynomial map $f\colon \mathbb{A}^n_k\to \mathbb{A}^n_k$ with an isolated zero at the origin \cite{KW-local-degree}. In equivariant homotopy theory, given a compact Lie group $G$, the degree of a proper $G$-equivariant map $f\colon X\to Y$ of smooth $G$-manifolds of the same dimension which is oriented relative to a genuine equivariant ring spectrum $E_G$ is a cohomology class in $E_G^0(Y)$ \cite{BW-G-degree}. Both the quadratically and equivariantly enriched global degrees can be written as a sum of local degrees, see \cite[Corollary 3.10]{KLSW-rational-curves} and \cite[Theorem 4.22]{BW-G-degree}. These enriched local degrees naturally lend themselves to computing enumerative invariants such as enriched Euler numbers and intersection multiplicities, proving their worth in enriched enumerative geometry. 

This paper seeks to take a step toward bridging the two programs by defining a degree, local degree, and Euler number in equivariant motivic homotopy theory for the purpose of studying equivariant enumerative problems over general fields. We construct these invariants using the six functor formalisms of Hoyois  and Khan-Ravi  \cite{Hoyois6ff, KR-generalized-cohomology}. The main application is a local multiplicity formula for an equivariant motivic B\'{e}zout's theorem. 

We briefly review the construction of our degree in the case of an equivariant map between $G$-schemes here. Let $k$ be a field and $G$ a finite group scheme over $k$ with $|G|$ coprime to $\op{char}(k)$. Let $E_G\in \SH^G(k)$ be an equivariant motivic ring spectrum. Given a quasi-smooth, proper, $G$-equivariant morphism $f\colon X\to Y$ between $G$-schemes $X$ and $Y$ of the same dimension, the equivariant motivic degree of $f$ is a cohomology class in $E_G^0(Y)$ so long as $f$ is relatively $E_G$-oriented. Specifically, a relative $E_G$-orientation of $f$ determines an equivalence $E_G^0(X) \xrightarrow{\simeq} E_G^0(X, L_f)$, allowing one to define an oriented pushforward $f_!^{\mathrm{or}}$,
\begin{center}
    \begin{tikzcd} 
    E_G^0(X, L_f)\arrow["f_!"]{r}  & E_G^0(Y) \\
    E_G^0(X)\arrow[swap, "f^{\mathrm{or}}_! "]{ur} \arrow [, "\simeq"]{u}& 
    \end{tikzcd} 
\end{center}
which is the composition of the equivalence $E_G^0(X)\simeq E_G^0(X, L_f)$ with the proper pushforward $f_!$. The degree is then  defined as $\deg^G f := f^{\mathrm{or}}_!(1_X)\in E_G^0(Y)$, where $1_X$ is the unit in $E_G^0(X)$.

With this setup, given a closed orbit $G \cdot y$ of a $\Spec k(y)$-valued point $y$ of $Y$ with stabilizer $G_y$, the relative $E_G$-orientation of $f$ pulls back to a relative $E_{G_x}$-orientation at any $x$ in the (derived) fiber $X_y \to \Spec k(y)$ with stabilizer $G_x$. 
For each classical point $x$ of the fiber, the open and closed
derived component $X_{y,x}$ supported at $x$ determines a local
degree. This component may have a nonreduced classical truncation
and may also carry genuinely derived structure (see Subsection \ref{subsection:degree_for_schemes} for details).
Using this we then define the local equivariant degree of any point in the fiber over $\Spec k(y)$ as an element of $E^0_{G_y}(\Spec k(y))$. 
Writing $i_y\colon \Spec k(y) \to Y$ for the inclusion of the point, and, for a subgroup $K \subset H$ over a scheme $S$, writing $\Tr_K^H$ for the pushforward $E^0_K(S) \to E^0_H(S)$ induced by $B_SK \to B_SH$, the first theorem of this paper is a local-to-global degree theorem.

\begin{theorem} \label{thm:degree_in_intro} Let $E_G \in \SH^G(k)$ be a motivic ring spectrum, and let $f\colon X\to Y$ be a proper, quasi-smooth, equivariantly smoothable $G$-equivariant map between $G$-schemes $X$ and $Y$ of the same dimension. Assume $f$ is relatively $E_G$-oriented.  
Let $i_y\colon \Spec k(y)\to Y$ be a point of $Y$ whose $G$-orbit
is closed and suppose that the classical truncation of the derived
fiber $X_y$ is zero-dimensional. 
Then 
\begin{equation}
i_y^* \deg^{G_y}f = \sum_{\{G_y\cdot x\subseteq X_y\}} \Tr_{G_x}^{G_y} \deg^{G_x}_x f 
\end{equation}
in $E^0_{G_y}(\Spec k(y))$.
\end{theorem}

See Subsection \ref{subsection:degree_for_schemes} and Theorem \ref{thm:local_global_for_schemes} for detailed constructions of the degree and local degree and the proof of the theorem. 

Unlike the topological equivariant degree and local degree of \cite{BW-G-degree}, Theorem \ref{thm:degree_in_intro} holds over fields other than the real and complex numbers. However, the degrees of Theorem \ref{thm:degree_in_intro} rely on the exceptional pushforward in equivariant motivic homotopy theory, whence cannot be applied to non-algebraic topological maps even over $\mathbb{R}$ and $\mathbb{C}$. We compare the construction of our global degree of a non-equivariant relatively oriented map to that of Kass-Levine-Solomon-Wickelgren, who use a different definition of relative orientation, in Remark \ref{rmk:degree_comparison} when $E$ is the Grothendieck-Witt sheaf and there is no group action. 

In Section \ref{section:bezout} we define the Euler number $n^E(\calV,\sigma)$ of a relatively $E$-oriented vector bundle $\calV$ on a derived stack $\calX$ in the cohomology of a motivic ring spectrum $E$ given a choice of section $\sigma$. In this level of generality, we take the definition of a motivic ring spectrum over $\calS$ to be that in \cite[Definition 9.3]{KR-generalized-cohomology}. 
We show the Euler number is independent of choice of section in Proposition \ref{prop:indep.-of-sec}, which we restate here: 

\begin{theorem}[Section independence]\label{thm:intro-section-indep} Given an $E$-oriented vector bundle $p: \calV \to \calX$ on a quasi-smooth, proper, representably smoothable stack $\pi_\calX \colon \calX \to \calS$ over a scalloped stack $\calS$, 
the Euler number $n^E(\calV,\sigma)$ is independent of the choice of the section $\sigma\colon \calX\to \calV$.
\end{theorem}

A key departure from \cite{BW-G-degree} and \cite{brazeuler} is that our computation of $n^E(\calV,\sigma)$ using local degrees requires the section to have isolated zeros, but the zeros do not necessarily need to be isolated and simple as in \cite{BW-G-degree} and \cite{brazeuler}. We expand on the computation of the Euler number as a sum of local degrees in the case when $\calX$ is the quotient of a scheme by a finite group $G$. Let $V := \mathbb{V}(\calE) \to X$ be a $G$-equivariant vector bundle over a proper, smooth $G$-scheme $X$ over $S = \Spec(k)$, with $\calE$ a $G$-equivariant locally free sheaf on $X$.
Let $\sigma$ be a $G$-invariant section. We first state the main result when $\sigma$ has zeros that are both isolated and simple: 

\begin{theorem} \label{intro_thm:Euler-number=local_deg}
Let $G$ be a finite group which is tame over $k$. Let $X$ be a proper, smooth algebraic space over $S= \Spec(k)$ 
    with $G$-action and $p: V \to X$ be an $E_G$-oriented equivariant vector bundle on $X$. Suppose $p$ has an equivariant section $\sigma$ with isolated and simple zeros $\coprod Gx_i$. 
    
    Let $E_G \in \SH^G(S)$ be any motivic ring spectrum. If $\sigma$ is compatibly $E_G$-oriented, then the equivariant Euler number with respect to $\sigma$ can be computed as the sum of local indices:
    \[ 
    n^{E_G}(V, \sigma) = \sum_{\{G\cdot x_i \subseteq Z(\sigma)\}} \Tr^G_{G_{x_i}} \Tr_{k(x_i)/k} (1_{G_{x_i}})
    \]
   in $E_G^0(S)$, where $G_{x_i} \subseteq G$ denotes the stabilizer group of $x_i \in Z(\sigma)$.
  By independence of section (Proposition~\ref{prop:indep.-of-sec}), this computes $n^{E_G}(V)$.
\end{theorem}

In other words, the equivariant Euler number can be computed as a sum over orbits of zeros of a section. See Definitions \ref{defn:group_transfer} and \ref{defn:field_transfer} for the definitions of $\Tr_{G_{x_i}}^G$ and $\Tr_{k(x_i)/k}$, and see Proposition \ref{prop:Euler-number=local_deg} for the proof. Note that the assumption that $p\colon V\to X$ is relatively $E_G$-oriented forces $\op{rank}V=\dim X$. 

When $\sigma$ is a section with zeros that are isolated but not necessarily simple,  
\[
n^{E_G}(V,\sigma) = \underset{\{G\cdot x_i\subseteq Z(\sigma)\}}{\sum} \Tr_{G_x}^G \Tr_{k(x)/k} \deg_x^{G_x}(f)
\]
where $f$ is constructed as follows. Let $U$ be a finite dimensional $G$-invariant $k$-vector subspace of $\Gamma(X,V)$ containing $\sigma$, for example we could take $U=\Gamma(X,V)^G$ or $\langle \sigma\rangle$. We also let $U$  denote the $G$-equivariant affine $k$-space defined by $U$. Let $Z$ be the universal zero locus of sections contained in $U$, defined by the derived pullback in the following homotopy cartesian square:
\begin{equation} 
  \begin{tikzcd}
        Z \arrow{r} \arrow{d}{i_Z} & X \arrow{d}{0}  \\
        U \times X \arrow{r}{ev} & V,
    \end{tikzcd}  
\end{equation}
where the bottom horizontal map is the evaluation map and the right vertical map is the zero section. 
We then define $f: Z \to U$ to be the map induced by the left vertical map $i_Z$ composed with the projection $p_U: U \times X \to U$. The  description of $f$ is given in full, along with orientation compatibility checks, in Construction \ref{construction:map_f_for_euler}. 


The Euler number allows us to define the degree of an intersection of closed substacks that is not necessarily transverse and compute it as a sum of local degrees. See related work of Khan for the construction of the virtual fundamental class of a derived, quasi-smooth intersection in general. In particular, if $\calZ_1$ and $\calZ_2$ are given by the derived vanishing loci of sections $s_1, s_2$ of  $E$-oriented vector bundles $\calV_1$ and $\calV_2$ over a proper, quasi-smooth, representably smoothable derived stack $\calX$ over a scalloped stack $\calS$, then the derived intersection $\calZ$ is the derived vanishing locus of the section $(s_1, s_2)$ of $\calV_1 \times_\calX \calV_2$. The Euler number 
\[
n^E(\calV_1 \times_\calX \calV_2)
\]
defines the $E$-intersection number of $\calZ_1$ and $\calZ_2$ in $\calX$. Since the Euler number is defined as a degree, see Definition \ref{defn:euler_number}, the $E$-intersection of $\calZ_1$ and $\calZ_2$ can be computed as a sum of local degrees. This extrapolates to the intersection of $\calZ_1, \ldots, \calZ_n$ given by the derived vanishing loci of sections $s_1, \ldots, s_n$ of $E$-oriented vector bundles $\calV_1, \ldots, \calV_n$. 

The culmination of computing 
an intersection number as an Euler number using degrees is an application to an enriched non-transverse B\'{e}zout's theorem. Enriched versions of B\'{e}zout's theorem exist, in the quadratically enriched case due to McKean \cite{mckean-bezout} and in the topological $C_2$-equivariant case due to Costenoble-Hudson-Tilson and Costenoble-Hudson \cite{CH-bezout, CH-bezout}. Our enriched B\'{e}zout's theorem allows us to compute the degree of an intersection defined by the simultaneous vanishing loci of sections. Like all results of this paper, our B\'ezout's theorem is purely algebro-geometric. We obtain an equivariantly enriched local multiplicity formula for B\'{e}zout's theorem for intersections that are isolated but not necessarily simple. 

\begin{theorem}[Equivariant non-transverse B\'ezout's theorem]  Let $k$ be a field, $G$ be a finite group that is tame over $k$, $S=\Spec(k)$, and  $E_G \in \SH^G(S)$ be an equivariant motivic ring spectrum. Let $X$ be a proper, smooth algebraic space with $G$-action over $S$. 
   Let $\calE_1, \ldots, \calE_n$ be $G$-equivariant locally free sheaves on $X$ over $S$ such that $\mathbb{V}(\oplus_{i=1}^n\calE_i)$ defines an $E_G$-oriented vector bundle on $X$ relative to $S$.  
   Let $\sigma_1, \ldots \sigma_n$ be $G$-invariant sections of $\calE_1, \ldots, \calE_n$, respectively.  Let $Z_1, \ldots, Z_n$ denote the derived vanishing loci of the sections $\sigma_1, \ldots, \sigma_n$ respectively. 
   
   The $E_G$-intersection number of $Z_1, \ldots, Z_n$ is given by $n^{E_G}(\mathbb{V}(\oplus_{i=1}^n\calE_i))$, and $n^{E_G}(\mathbb{V}(\oplus_{i=1}^n\calE_i))$ is independent of choice of $G$-invariant sections $\sigma_1, \ldots, \sigma_n$. If $\sigma$ is a section of $\mathbb{V}(\oplus_{i=1}^n\calE_i)$  with isolated, but not necessarily simple, zeros $Z(\sigma)=\{x_1, \cdots, x_k\}$, then the $E_G$-intersection number can be computed as
   \[
   \sum_{\{G\cdot x_i\subseteq Z(\sigma)\}} \Tr^G_{G_{x_i}} \deg_{x_i}^{G_{x_i}}f
   \]
   in $E^0_G(S)$, where $f$ is as in Construction \ref{construction:map_f_for_euler} and $G_{x}$ denotes the $G$-stabilizer at a point $x \in X$. When $Z(\sigma)$ consists of zeros that are both isolated and simple and $\sigma$ is compatibly $E_G$-oriented, the $E_G$-intersection number can be computed as
   \[
   \sum_{\{G\cdot x_i\subseteq Z(\sigma)\}} \Tr^G_{G_{x_i}} \Tr_{k(x_i)/k} (1_{G_{x_i}})
   \]
   in $E^0_G(S)$.  
   \end{theorem}

This should be interpreted as a local multiplicity formula in the case of $G$-schemes for equivariant intersections that are isolated but not necessarily transverse. The intersection number defined here is the degree of the cohomological fundamental class of Khan's non-transverse B\'ezout theorem for $G$-schemes when both apply, see \cite[Theorem 3.22]{khan-VFC}. Note Khan's non-transverse B\'ezout theorem for intersection classes works in more general settings than ours. 

\noindent{\bf Conventions.}
We assume that all (derived) stacks are quasi-compact and quasi-separated. 
All stacks are assumed to be homotopically smooth in the sense that the cotangent complex is a perfect complex over the base. 

We define the notion of a motivic ring spectrum $E$ over the base (scalloped) stack $\calS$ as in \cite[Section 9]{KR-generalized-cohomology}. Specifically, a \textit{motivic ring spectrum} $E$ over $\calS$ is a collection of objects $E_\calX \in \SH(\calX)$ for every scalloped derived stack $\calX$ representable over $\calS$, together with a homotopy coherent system of isomorphisms $$f^*E_\calY\simeq E_\calX$$ in $\SH(\calX)$ for every representable morphism $f\colon \calX\to \calY$. Given a motivic ring spectrum $E$ over $\calS$, we define $E$-cohomology in Definition \ref{defn:cohomology}. Examples include homotopy invariant $K$-theory $\KGL$, motivic cobordism $\MGL$, and motivic cohomology. See \cite[Section 5]{KR-generalized-cohomology} to reference the definition of $\SH$. 

Given a stack $\calX$ representable over a base $\calS$, $L_\calX$ is always meant to denote the cotangent complex of $\calX$ over $\calS$, cohomologically graded in $(-\infty, 0]$.
Given a locally free sheaf $\calE$ on $\calX$,  $\mathbb{V}(\calE) := \Spec \op{Sym} \calE$.  The \emph{Thom spectrum} of $\calE$ in $\SH(\calX)$ is defined here to be the (stabilization or) infinite suspension spectrum of $\mathbb{V}(\calE)/(\mathbb{V}(\calE) \setminus \calX)$. The Thom spectrum is typically denoted $\Sigma^{\infty}\Th_\calX(\calE)$ to distinguish between Thom spectra and their unstable counterparts, but we reserve $\Th_\calX(\calE)$ for the spectrum since we only work stably in this paper. This definition extends to perfect complexes on a scalloped derived stack $\calX$ by \cite[Theorem 4.5.i]{KR-generalized-cohomology}, which says $\Th_\calX(\calE) \in \SH(\calX)$ is $\otimes$-invertible, and the assignment $\calE\mapsto \Th_\calX(\calE)$ induces a canonical map 
\[
\Th_\calX\colon K(\calX)\to \op{Pic}(\SH(\calX))
\]
from the algebraic $K$-theory of perfect complexes on $\calX$ to $\otimes$-invertible objects in $\SH(\calX)$. See also \cite[Section 16]{Bachmann2017NormsIM} for more on Thom spectra and the motivic Thom spectrum functor. 

Finally, for any finite locally free sheaf $\calE \to \calX$ there is an endofunctor $$\langle \calE\rangle\colon \SH(\calS)\to \SH(\calS),$$ called the \emph{Thom twist} by $\calE$, which is characterized by the equivalence $\langle\calE\rangle \simeq (-)\otimes \Th_{\calX}(\calE)$. See \cite[Definition 5.4, Example 5.12.iv]{KR-generalized-cohomology}. For $\calE\in K_0(\calX)$, we may denote $(-)\otimes\Th_\calX(\calE)$ by $\Sigma_\calX^\calE$ when it is notationally convenient to do so. 

Finally, we note that all pullbacks are derived unless otherwise noted.

\noindent{\bf Acknowledgements.} Candace Bethea was supported by National Science Foundation award DMS-2402099. She would like to thank the Tata Institute of Fundamental Research for providing a productive working environment while the majority of this work was being conducted. She would also like to thank Kirsten Wickelgren for several helpful discussions related to this work. Charanya Ravi thanks Adeel Khan for numerous discussions on equivariant six-functor formalisms and for his collaboration on the paper \cite{KR-generalized-cohomology}, which provides the framework required for this work. She acknowledges the support provided by ANRF PM Early Career Research Grant ANRF/ECRG/2024/001733/PMS and the Department of Atomic Energy, Government of India, under Project No. RTI 4014. 
We thank the anonymous referee of an earlier version of this paper for careful reading and for many helpful comments, corrections, and suggestions.
We warmly thank the Simons Laufer Mathematical Sciences Institute (SLMath), supported by the National Science Foundation Grant no. DMS-2424139, for providing a welcoming and productive environment where this work was conducted in part while the authors were in residence at SLMath in Berkeley, California during the Fall 2026 semester.

\section{Degree and Local Degree}\label{section:degree}

In this section, we present the degree and local degree for quasi-smooth, proper, representably smoothable equivariant morphisms between $G$-schemes over a field $k$ of the same dimension for a finite group $G$ of order coprime to the exponential characteristic of $k$, $\op{char}(k)$. This relies on a general base change construction that applies to a large class of stacks called \emph{scalloped stacks} (Definition \ref{defn:scalloped}) due to Khan-Ravi \cite{KR-generalized-cohomology}. The reader new to scalloped derived stacks should think of these as being stratified by quotient stacks that are the quotient of a derived quasi-affine scheme by a finite \'{e}tale extension of a multiplicative type group scheme. This includes classical Deligne-Mumford or Artin stacks with separated diagonal which are tame \cite[Examples 2.15-2.16]{KR-generalized-cohomology}, where tame is meant in the sense of Abramovich-Olsson-Vistoli \cite{AOV-tame-stacks}.

In Subsection \ref{subsection:degree_for_scalloped}, we review the general construction, and in Subsection \ref{subsection:degree_for_schemes} we present the degree and local degree for $G$-schemes. This is an algebraic equivariant analogue of the topological equivariant degree and local degree of Bethea-Wickelgren \cite[Sections 3-4]{BW-G-degree}, defined using genuine stable equivariant homotopy theory. 

\subsection{General base change}\label{subsection:degree_for_scalloped}

We recount first the relevant definitions of scalloped stacks and generalized cohomology from \cite{KR-generalized-cohomology}.

\begin{definition}\label{defn:nice_embeddable}
    \cite[Definition 2.1]{KR-generalized-cohomology} Let $G$ be an fppf affine group scheme over $S$ an affine scheme. 
    \begin{enumerate}
        \item We say $G$ is \emph{nice} if it is an extension of a finite \'{e}tale group scheme, of order prime to the residue characteristics of $S$, by a group scheme of multiplicative type. 
        \item We say $G$ is \emph{embeddable} if it is a closed subgroup of $GL_S(\mathcal{E})$ for some locally free sheaf $\mathcal{E}$ on $S$. 
    \end{enumerate}
\end{definition}

\begin{definition}\label{defn:scallop_decomposition}
    \cite[Definition 2.7]{KR-generalized-cohomology} Let $\mathcal{X}$ be a derived algebraic stack. A \emph{scallop decomposition} $(\mathcal{U}_i, \mathcal{V}_i, u_i)$ of $\mathcal{X}$ is a finite filtration by quasi-compact open substacks 
    \[
    \varnothing = \mathcal{U}_0 \hookrightarrow \mathcal{U}_1 \hookrightarrow \cdots\hookrightarrow \mathcal{U}_n = \mathcal{X},
    \]
    together with Nisnevich squares 
    \[
    \begin{tikzcd}
        \mathcal{W}_i \arrow{r} \arrow{d} & \mathcal{V}_i \arrow["u_i"]{d} \\
        \mathcal{U}_{i-1} \arrow{r} & \mathcal{U}_i
    \end{tikzcd}
    \]
    where $u_i$ are representable \'{e}tale morphisms of finite presentation. 
\end{definition}

\begin{definition}\label{defn:scalloped}
\cite[Definition 2.9]{KR-generalized-cohomology} Let $\mathcal{X}$ be a quasi-compact and quasi-separated derived algebraic stack. 
\begin{enumerate}
    \item $\mathcal{X}$ is \emph{nicely fundamental} if it admits an affine morphism $\mathcal{X}\to BG$ for some nice embeddable group scheme $G$ over an affine scheme $S$. That is, $\mathcal{X}$ is the quotient $[X/G]$ of an affine derived scheme with a $G$-action. 
    \item $\mathcal{X}$ is \emph{nicely quasi-fundamental} if it admits a \emph{quasi-affine} morphism $\mathcal{X}\to BG$ for some nice embeddable group scheme $G$ over an affine scheme $S$. That is, $\mathcal{X}$ is the quotient $[X/G]$ of a quasi-affine derived scheme $X$ over $S$ with a $G$-action. 
    \item $\mathcal{X}$ is \emph{nicely scalloped} if it has separated diagonal and admits a scallop decomposition $(\mathcal{U}_i, \mathcal{V}_i, u_i)$ where $\mathcal{V}_i$ are quasi-fundamental. 
\end{enumerate}
\end{definition}

The adjective `nicely' is dropped henceforth. Scalloped stacks have a description in terms of stabilizers: 

\begin{theorem}\label{thm:scalloped_nice_stabilizers}
    \cite[Theorem 2.12]{KR-generalized-cohomology} 
    \begin{enumerate}
    \item The classes of fundamental, quasi-fundamental, and scalloped derived stacks are stable under finite disjoint unions. 
    \item Let $\mathcal{X}$ be a qcqs derived algebraic stack. If $\mathcal{X}$ has separated diagonal, then the following are equivalent: \label{thm:sep_diagonal_scalloped}
    \begin{enumerate}
        \item $\mathcal{X}$ admits a Nisnevich cover $u\colon \mathcal{U}\twoheadrightarrow  \mathcal{X}$ where $\mathcal{U}$ is fundamental.
        \item $\mathcal{X}$ is scalloped.
        \item $\mathcal{X}$ admits a scallop decomposition $(\mathcal{U}_i, \mathcal{V}_i, u_i)_i$ where the $\mathcal{V}_i$ are scalloped.
        \item $\mathcal{X}$ has nice stabilizers. 
    \end{enumerate}
    \item Let $\mathcal{X}$ be a qcqs derived algebraic stack. If $\mathcal{X}$ has affine diagonal, then the following are equivalent: 
    \begin{enumerate}
        \item $\mathcal{X}$ is scalloped. 
        \item $\mathcal{X}$ admits an affine Nisnevich cover $u\colon \mathcal{U}\twoheadrightarrow  \mathcal{X}$ where $\mathcal{U}$ is fundamental.
        \item $\mathcal{X}$ admits a scallop decomposition $(\mathcal{U}_i, \mathcal{V}_i, u_i)_i$ where the $\mathcal{V}_i$ are quasi-fundamental and the $u_i$ are quasi-affine.
    \end{enumerate}
    \end{enumerate}
\end{theorem}

\begin{remark}
    Any derived stack which is representable over a scalloped stack is scalloped. 
\end{remark}

\begin{assumption}
    Throughout, we will always assume that any scalloped stack has affine diagonal to allow for the construction of Gysin transformations. 
\end{assumption}

 We will work parameterized over a base throughout this paper. Let $\calS$ be a scalloped derived stack, and let $\pi_\calX \colon \calX\to \calS$ and $\pi_\calY \colon \calY\to\calS$ be derived scalloped stacks representable over $\calS$. 
 Let $E$ be a motivic ring spectrum over $\calS$.  Twisted cohomology valued in $E$ is defined as follows, following \cite[Section 9]{KR-generalized-cohomology} and \cite[Section 2.2]{DJK} in the case of schemes: 

\begin{definition}\label{defn:cohomology} (See \cite[Definition~2.2.1]{DJK}, \cite[Section~9.1]{KR-generalized-cohomology})
    Suppose $\pi_\calX\colon \calX \to \calS$ is representable.
    Let $V$ be a class in $K_0(\calX)$. The \emph{$V$-twisted $E$-cohomology of $\calX$} is the mapping spectrum
    \[
    E(\calX, V):=\op{Maps}_{\SH(\calX)}(\unit_\calX, \pi_\calX^*E\otimes\Th_\calX (V)),
    \]
    for $V = 0$, we will denote $E(\calX,V)$ by $E(\calX)$.
    Note that $V$-twisted $E$-cohomology of $\calX$ can be equivalently defined as  $\op{Maps}_{\SH(\calS)}(\unit_\calS,{\pi_\calX}_*(\pi_\calX^*E\otimes \Th_\calX(V))$ by adjunction. We will write
    \[
    E^n(\calX, V):= \pi_{-n} E(\calX,V). 
    \]
    For a finite type, representable map $f: \calX \to \calY$ over $\calS$,  the \emph{$V$-twisted bivariant spectrum} of $\calX$ over $\calY$ is
    \[
     E(\calX/\calY, V) := \Maps_{\SH(\calX)}(\unit_\calX, \Sigma^{-V}_\calX f^!\pi_\calY^*E).
    \]
    We similarly define $E(\calX/\calY):= E(\calX/\calY,0)$ and $E_n(\calX/\calY, V) := \pi_n E(\calX/\calY, V)$. Note that $E(\calY/\calY, V) = E(\calY,V)$ for the identity map $\calY \to \calY$. When $V$ is the class of a perfect complex or sheaf $\calE$, we will use the notation $V$ and $\calE$ interchangeably.

    \end{definition}

\begin{definition}\label{defn:quasi_smooth}
    A representable morphism of derived stacks is \emph{quasi-smooth} if it is locally of finite presentation and the relative cotangent complex is perfect of Tor-amplitude $[-1,0]$ (with cohomological grading).
\end{definition}

When $f\colon X\to Y$ is a map of classical stacks, or in particular classical schemes, $f$ being quasi-smooth is equivalent to $f$ being lci. 
 
\begin{definition}\label{defn:smoothly_rep}
    (\cite[Definition 8.3]{KR-generalized-cohomology}) A morphism $f\colon \calX\to \calY$ of scalloped derived stacks is \emph{representably smoothable} if it admits a global factorization 
    \[
    \calX \stackrel{i}{\longrightarrow} \mathcal{A} \stackrel{p}{\longrightarrow} \calY
    \]
    where $p$ is smooth representable and $i$ is finite unramified. 

    In the particular case where $f$ is the map on quotient stacks induced by a $G$-equivariant morphism $f': X \to Y$ of algebraic spaces with $G$-action, $f$ is representably smoothable if and only if $f'$ admits a $G$-equivariant global factorization as a finite unramified morphism followed by a smooth morphism of $G$-spaces. The morphism $f'$ is said to be \emph{equivariantly smoothable} in that case.
\end{definition}

    Assume $f\colon \calX \to \calY$ is quasi-smooth, proper, representably smoothable morphism \cite[Definition 8.3]{KR-generalized-cohomology} and $L_f$ has constant virtual rank 0. There is a Gysin pushforward in cohomology, 
\[
f_!\colon E(\calX,L_f)\to E(\calY). 
\]
Explicitly, writing $\pi_\calY\colon \calY\to \calS$ for the structure map of $\calY$ so that $\pi_\calX = \pi_\calY f$, the proper pushforward is given by
\begin{align} \label{eq:Gysin-push}
    E(\calX, L_f) &= \Maps_{\SH(\calX)}(\unit_\calX, \Sigma^{L_f}_\calX\pi_\calX^*E) \\ \notag
    &\simeq^1 \Maps_{\SH(\calS)}(\unit_\calS , {\pi_\calX}_* \Sigma^{L_f}_\calX \pi_\calX^* E) \\ \notag
    &\simeq \Maps_{\SH(\calS)}(\unit_\calS , {\pi_\calY}_*f_*\Sigma^{L_f}_\calX f^* \pi_\calY^* E) \\ \notag
    &\simeq^2 \Maps_{\SH(\calY)}(\unit_\calY, f_!\Sigma^{L_f}_\calX f^* \pi_\calY^*E) \\ \notag
    &\xrightarrow{\mathrm{gys}_f} \Maps_{\SH(\calY)}(\unit_\calY, f_!f^! \pi_\calY^*E) \\ \notag
    &\to^3 \Maps_{\SH(\calY)}(\unit_\calY, \pi_\calY^*E) = E(\calY),
\end{align}
where $\simeq^1$ is given by the adjunction $\pi_\calX^* \dashv \pi_{\calX,*}$, $\simeq^2$ is given by the equivalence $f_!\simeq f_*$ for $f$ proper and representable, see \cite[Theorem~7.1(ii)]{KR-generalized-cohomology}, $\mathrm{gys}_f: \Sigma^{L_f}_\calX f^* \to f^!$ is the Gysin transformation for the quasi-smooth, representably smoothable
morphism $f$ constructed in \cite[Theorem~8.4]{KR-generalized-cohomology}
and $\to^3$ is given by the counit of the adjunction $f_! \dashv f^!$. For any class $V \in K_0(\calX)$, the pushforward is 
\[
f_!\colon E(\calX, L_f+V)\to E(\calY, V). 
\]
Note that even when $f$ is a map between classical stacks, or even schemes, the Gysin pushforward could still be virtual. 

We would like to define an $E$-valued  degree as the pushforward of the Poincar\'{e} dual of the fundamental class of $\calX$, as in the topological equivariant degree in \cite[Definition 3.5]{BW-G-degree}. However, as in the topological case, we will be interested in an untwisted pushforward, $E(\calX)\to E(\calY)$ rather than $f_!\colon E(\calX,L_f)\to E(\calY)$, by accounting for a relative $E$-orientation that trivializes $L_f$ with respect to $E$.

\begin{definition}\label{defn:rel_oriented}
A \emph{relative $E$-orientation} of a map $f: \calX \to \calY$ is an equivalence of motivic spectra
\[
\rho_f\colon E_\calX \stackrel{\sim}{\longrightarrow} E_\calX \otimes \Th_\calX(L_f),
\]
where $E_\calX:= \pi_\calX^*E$ for $\pi_{\calX}\colon \calX\to \calS$ the structure map. We will say $f$ is \emph{relatively $E$-oriented} if there is a choice of an equivalence $E_\calX \simeq E_\calX \otimes \Th_\calX(L_f)$. Note that the specific choice of equivalence is part of the orientation data. 
A relative orientation as above immediately induces an equivalence in $E$-cohomology 
\[
\rho_f\colon E(\calX) \stackrel{\sim}{\longrightarrow} E(\calX,L_f),
\]
which we also denote by $\rho_f$ by abuse of notation. The induced equivalence at the level of the cohomology spectra, or even simply the $0^{\text{th}}$ cohomology groups, induced by the orientation that is used for most of the constructions are all denoted by $\rho_f$.

For $E$ a motivic ring spectrum over $\calS$, we say that \emph{$\calX$ is $E$-oriented} if the structure morphism $\calX \to \calS$ is relatively $E$-oriented, and we denote the orientation by $\rho_\calX$. A morphism $f: \calX \to \calY$ between $E$-oriented stacks is said to be a \emph{compatibly oriented} by an equivalence $\rho_f$ if the composite
\[
E_\calX \xrightarrow{\rho_f} E_\calX \otimes \Th_\calX(L_f) \xrightarrow{f^*\rho_\calY} f^*E_\calY \otimes \Th_\calX (f^* L_{\calY/\calS}) \otimes \Th_\calX(L_f)
\]
is homotopic to $\rho_\calX$, under the identification $E_\calX \simeq f^* E_\calY$ and $$\Th_\calX(L_{\calX/\calS}) \simeq \Th_\calX(f^*L_{\calY/\calS}) \otimes \Th_\calX(L_f)$$ given by  $L_f = L_{\calX} - f^*L_{\calY}$ in $\K_0(\calX)$.
\end{definition}

\begin{remark} \label{lem:2-out-of-3}
Let $\calS$ be a scalloped stack and let $E$ a motivic ring spectrum over $\calS$.
Given a representable morphism $f: \calX \to \calY$ between $E$-oriented, almost finitely presented derived stacks, representable over $\calS$ with respective orientations given by $\rho_\calX$ and $\rho_\calY$, we can always define a compatible relative $E$-orientation of $f$ as the composite
     \[
     (f^* \rho_\calY)^{-1} \circ \rho_\calX: E_\calX \to E_\calX \otimes \Th_\calX(L_{\calX/\calS}) \to  E_\calX \otimes \Th_\calX(L_f).
     \]
We call this the induced compatible relative orientation of the morphism $f$.
\end{remark}

\begin{remark}
A representable, étale morphism $f$ is canonically relatively oriented for any motivic ring spectrum $E$ over $\calS$ since $L_f \simeq 0$, and we always assume this equivalence to be the relative orientation unless specified otherwise. A relative $E$-orientation of a morphism $f$ defines a relative orientation on any derived base change of $f$ by pullback. 
\end{remark}

The morphism $f$ being relatively $E$-oriented as in Definition \ref{defn:rel_oriented} implies there is a choice of pushforward $f_!^{\mathrm{or}}:E(\calX)\to E(\calY)$
\begin{equation}\label{eq:oriented_pushforward}
    \begin{tikzcd} 
    E(\calX, L_f)\arrow["f_!"]{r}  & E(\calY) \\
    E(\calX)\arrow[swap, "f^{\mathrm{or}}_! "]{ur} \arrow [, "\rho_f"]{u}& 
    \end{tikzcd} 
\end{equation}
that factors through $f_!\colon E(\calX, L_f)\to E(\calY)$. Though we work with  motivic ring spectra and scalloped stacks rather than genuine $G$-spectra and topological $G$-manifolds, Definition \ref{defn:rel_oriented} otherwise mirrors  \cite[Definition 3.1]{BW-G-degree}. Note that a choice of orientation, hence a choice of pushforward $f_!^{\mathrm{or}}\colon E(\calX)\to E(\calY)$, may not be canonical in general.

\begin{remark} In practice, one is typically interested in specific ring spectra, such as homotopy invariant K-theory $\KGL$, algebraic cobordism $\MGL$, and motivic cohomology $\HZ$, for computational purposes. See \cite[Section 10]{KR-generalized-cohomology} for the constructions of these theories for scalloped stacks. 
\end{remark}

Accounting for a relative $E$-orientation, we can define the degree of $f\colon \calX\to \calY$ as follows. 

\begin{definition}\label{defn:global_degree}
    Let $E$ be a motivic ring spectrum over $\calS$, and let $f\colon \calX\to \calY$ be a relatively $E$-oriented, proper, quasi-smooth,  representably smoothable morphism between smooth, proper derived scalloped stacks over $\calS$ such that $L_f$ has constant virtual rank 0.  Denote the choice of relative orientation by $\rho_f\colon E(\calX)\stackrel{\sim}{\to} E(\calX, L_f)$. Recall the oriented pushforward $f^{\mathrm{or}}_!$ of $f$, see \eqref{eq:oriented_pushforward}. The \emph{$E$-degree} of $f$ is $f_!^{\mathrm{or}}(1_\calX)\in E^0(\calY)$, denoted $\deg f$. Equivalently, $f_{!}(\rho_{f}1_\calX)$ in $E^0(\calY)$. \end{definition}

\begin{notation} \label{notation:degree_of_stack} Given $\calX$ over $\calS$ which is $E$-oriented with structure map $\pi_{\calX}\colon \calX\to \calS$, we will write $\deg(\calX)$ for $\deg \pi_{\calX} = \pi_{\calX,!}^{\mathrm{or}}(1_{\calX}) \in E(\calS)$ and call this the $E$-\emph{degree of $\calX$}. 
\end{notation}

    Equivalently, the $E$-degree of $f\colon \calX\to \calY$ could be defined as $f_*(\rho_f(1_\calX) \cap [\calX/\calY])$, where
    $[\calX/\calY] \in E_0(\calX/\calY, -L_f)$ is the relative fundamental class (see \cite[Definition~9.13]{KR-generalized-cohomology}), and
    \[
    f_*: E_0(\calX/\calY) \rightarrow E_0(\calY/\calY) = E^0(\calY)
    \]
    is the pushforward for Borel-Moore homology relative to $\calY$,
    when the relative fundamental class exists. See \cite[Definition 9.14]{KR-generalized-cohomology}.

Before defining the local degree, we first remark on the comparison between the Definition \ref{defn:global_degree} and the degree obtained as an oriented pushforward in the quadratically enriched setting \cite{KLSW-rational-curves}.

\begin{remark}\label{rmk:degree_comparison}
    The reader familiar with enriched enumerative geometry may wonder how the relative orientation of Definition \ref{defn:rel_oriented} and oriented pushforward $f_!^{\mathrm{or}}$ compares to the quadratically enriched setting. Suppose $f\colon X\to Y$ is a finite, flat, local complete intersection morphism of schemes over a field. Kass-Levine-Solomon-Wickelgren define the quadratically enriched degree of such a map, which is valued in $\mathcal{GW}(Y)$, the Grothendieck--Witt sheaf of non-degenerate symmetric bilinear forms on $Y$ \cite[Definition 2.4]{KLSW-rational-curves}. Their degree depends on a choice of relative orientation for $\omega_f:=\op{det}L_f$. We point out that a quadratically enriched relative orientation in the sense of loc. cit. differs from Definition \ref{defn:rel_oriented}, it instead asks that $\omega_f \cong L^2$ for some line bundle $L\to X$. However, a relative orientation in the sense of loc. cit. gives a relative orientation in the sense of Definition \ref{defn:rel_oriented} when both are defined. 
    
    Suppose $f\colon X\to Y$ is a non-equivariant finite, flat, local complete intersection morphism between schemes of the same dimension over a field $k$.Suppose $f$ is relatively oriented in the sense of Kass--Levine--Solomon--Wickelgren, i.e., suppose there exists some line bundle $L\to X$ and choice of isomorphism $ \omega_f \stackrel{\rho_f}{\simeq} L^{\otimes 2}$, where $\omega_f:=\op{det}L_f$. Let $\mathcal{GW}$ be the Grothendieck--Witt sheaf, so $\mathcal{GW}(Y)$ has sections given by isomorphism classes of non-degenerate, symmetric bilinear forms $V \otimes V\to \mathcal{O}_Y$ for locally free $V$ on $Y$. Note $1_X\in \mathcal{GW}(X)$ is given by the multiplication $\calO_X\otimes \calO_X \to \calO_X$, which is equivalent to $\Hom_{\calO_X}(\calO_X,\mathcal{H}om_{\calO_X}(\calO_X,\calO_X))$ by adjunction. Tensoring with the orienting line bundle $L$, by adjunction and the isomorphism $\rho_f$ there is an equivalence 
    \begin{align*}
    \Hom_{\calO_X}(\calO_X, \calHom_{\calO_X}(\calO_X,\calO_X)) & \stackrel{\otimes L}{\cong} \Hom_{\calO_X}(L, \calHom_{\calO_X}(L, L^{\otimes 2})) \\ 
    &\cong \Hom_{\calO_X}(L, \calHom_{\calO_X}(L, \omega_f)) = \rho_f(1_X) \in \mathcal{GW}(X,\omega_f). 
    \end{align*}
    Pushing this forward and composing with the trace $\op{Tr}_f$ from coherent duality, we obtain 
    $$
    \Hom_{\calO_X}(L, \calHom_{\calO_X}(L, \omega_f)) \to \Hom_{\calO_Y}(f_*L, \calHom_{\calO_Y}(f_*L, \calO_Y)), 
    $$
    In total, the quadratically enriched degree of Kass--Levine--Solomon--Wickelgren is a bilinear form on $\calO_Y$, 
    $$
    f_*L\otimes f_*L\to f_* L^{\otimes 2} \stackrel{f_*\rho_f}{\simeq} f_*\omega_f \stackrel{\op{Tr}_f}{\to} \calO_Y,$$
    which can be described using the above construction. 
    This is equivalent to applying $f_!$ to $\rho_f(1_X)$ by Grothendieck--Serre duality. 
Thus the quadratic relative orientation definition of \cite[2.1]{KLSW-rational-curves} induces a relative orientation in the sense of Definition \ref{defn:rel_oriented}, and that the resulting degree $\deg(f) = f^{\mathrm{or}}_!(1_X)$
    agrees with \cite[Definition 2.4]{KLSW-rational-curves}.
\end{remark}

We now define the local degree and show that the degree can be computed as a sum of local degrees. See \cite{KW-local-degree, KLSW-rational-curves, BW-G-degree} for examples of local to global degree theorems in quadratic and (topological) equivariantly enriched contexts for suitable schemes and $G$-manifolds respectively. 

Let $\calZ \subseteq \calY$ be a closed derived substack of $\calY$. 
There is a homotopy cartesian square of scalloped stacks 
\begin{equation}\label{diagram:fiber_over_point}
\begin{tikzcd}
\calX_\calZ \arrow[swap, d, "f'"] \arrow[r, hookrightarrow, "i_{\calX_{\calZ}}"] \arrow[dr, phantom, "\lrcorner_{\text{ }h}", very near start] & \calX \arrow[d,"f"]\\
\calZ \arrow[swap, r, hookrightarrow, "i"] & \calY
\end{tikzcd}
\end{equation}
where $\calX_{\calZ}$ is the pullback and $f'$ is $f|_{\calX_\calZ}$. 

For any closed (note it is also open) component $\calZ'$ of $\calX_\calZ$, we can restrict $f'$ in \eqref{diagram:fiber_over_point} to $\calZ'$ to obtain 
\begin{center}
    \begin{tikzcd}
        \calZ' \arrow[hookrightarrow, r,"i_{\calZ'}" ] \arrow[swap, d, "f_{\calZ'}"] & \calX \arrow[d,"f"]\\
        \calZ \arrow[hookrightarrow, swap, r, "i"] & \calY
    \end{tikzcd}
\end{center}
where $f_{\calZ'}$ is $f'|_{\calZ'}$, though the above square is no longer cartesian. By pullback, $f_{\calZ'}$ is oriented by $i_{\calZ'}^*\rho_f \colon E(\calZ')\stackrel{\sim}{\to} E(\calZ', L_{f_{\calZ'}})$, denote this orientation by $\rho_{f_{\calZ'}}$. As was the case for $f$, there is an oriented pushforward $f_{\calZ'!}^{\mathrm{or}}$ in $E$-cohomology, defined as the composite
\begin{center}
\begin{tikzcd}
E(\calZ') \arrow[r, "\rho_{f_{\calZ'}}"] \arrow[rr, bend right=15, "f_{\calZ'!}^{\mathrm{or}}"'] &  E(\calZ', L_{f_{\calZ'}})\arrow[r,"f_{\calZ'!}"] & E(\calZ)
\end{tikzcd}
\end{center}

Note that $f'$ and $f_{\calZ'}$ are quasi-smooth, proper, and representably smoothable since $f$ is. 

\begin{definition}\label{defn:local_degree}
    The local degree of $f$ at $\calZ'\subseteq \calX_\calZ$ is $\deg_{\calZ'} f:=(f_{\calZ'})_!^{\mathrm{or}}(1_{\calZ'})\in E^0(\calZ)$. Equivalently, $f_{\calZ',!}(\rho_{f_{\calZ'}}1_{\calZ'})$ in $E^0(\calZ)$. 
\end{definition}

The local to global relation between the degree of $f$ and the sum of local degrees over the fiber of $\calZ$ is immediate by base change:

\begin{proposition}\label{prop:local_to_global}
    Let $E$ be a motivic ring spectrum over $\calS$, and let $f\colon \calX \to \calY$ be a quasi-smooth, proper, representably smoothable morphism of derived scalloped stacks that is relatively $E$-oriented, and suppose $L_f$ has constant virtual rank 0. Let $i\colon \calZ\hookrightarrow \calY$ be a closed derived substack of $\calY$. Then 
\begin{equation}\label{eq:deg_eq}
i^* \deg f = \sum_{\calZ'\subseteq \calX_{\calZ}} \deg_{\calZ'} f
\end{equation}
in $E^0(\calZ)$, where the sum is taken over all components $\calZ'$ in $\calX_{\calZ}$. 
\end{proposition}

\begin{proof}
    Given such a $\calZ\subseteq \calY$ and any $\calZ'\subseteq \calX_\calZ$, the following diagram commutes
\begin{equation}\label{diag:local_to_global_diag}
\begin{tikzcd}
\calZ'\arrow[hookrightarrow, r,"i_{\calZ'}" ] \arrow[swap, bend right=20, dr, "f_{\calZ'}"] & \calX_\calZ \arrow[swap, d, "f'"] \arrow[r, hookrightarrow, "i_{\calX_{\calZ}}"] \arrow[dr, phantom, "\lrcorner_{\text{ }h}", very near start] & \calX \arrow[d,"f"]\\
 & \calZ \arrow[swap, r, hookrightarrow, "i"] & \calY
\end{tikzcd}
\end{equation}
where the right square is homotopy cartesian and $f_{\calZ'} = f'|_{\calZ'}$. By additivity of $E$-cohomology (see \cite[Definition 5.2 and Example 5.12]{KR-generalized-cohomology}), it is immediate that
\begin{align*}
i^*\deg(f) & = i^*f_!^{\mathrm{or}}(1_\calX) \\
&= (f')_!^{\mathrm{or}}(i_{\calX_{\calZ}})^*(1_\calX) \\
&= (f'_!)^{\mathrm{or}}(1_{\calX_\calZ}) \\
&= \sum_{\calZ'\subseteq\calZ} i_{\calZ',*}f_{\calZ'!}^{\mathrm{or}}(1_{\calZ'})\\
&= \sum_{\calZ'\subseteq\calZ} \deg f_{\calZ'},
\end{align*}
where $i^*f_!(1_\calX) = (f')_!(i_{\calX_{\calZ}})^*(1_\calX)$ by proper base change for quasi-smooth morphisms \cite[Theorem 6.1.ii]{KR-generalized-cohomology}. This finishes the proof since $\deg_{\calZ'} f = \deg f_{\calZ'}$ by definition. 
\end{proof}

Proposition \ref{prop:local_to_global} is immediate by proper base change for any sufficiently nice $\calZ\subseteq\calY$, but is most interesting when $\calZ$ is a closed point of a scheme or particularly nice stack (e.g., tame Deligne-Mumford). Even if $\calZ$ is a closed point of a scheme and $f$ is a smooth, proper map between smooth $G$-schemes for finite $G$, Equation \eqref{eq:deg_eq} will not be independent of choice of the point $\calZ$, as one might choose points of the target with different stabilizers. This is also true in the topological case \cite{BW-G-degree}. The case of the equivariant degree and local degree of a map of $G$-schemes for finite $G$ is presented in Subsection \ref{subsection:degree_for_schemes}. 

Proposition \ref{prop:local_to_global} is also interesting when $\calX$ and $\calY$ are schemes and $\calZ\subseteq\calY$ is a closed point but $i_{\calZ}\colon \calZ \hookrightarrow \calY$ and $f\colon \calX \to \calY$ are not transverse. Only requiring the fiber $\calX_\calZ$ to be quasi-smooth rather than smooth is a departure from the local to global degree formulation of Bethea-Wickelgren \cite{BW-G-degree}, and ultimately will allow us to compute Euler numbers for vector bundles over smooth, projective schemes with respect to a section that may not be transverse to the zero section. This is expanded upon in Subsection \ref{section:bezout}.

We end this section by stating a general form of oriented base change for later reference, which we don't prove. 

\begin{proposition}\label{prop:oriented_base_change}
    Let $E$ be a motivic ring spectrum over $\calS$, and let $f\colon \calX \to \calY$ be a quasi-smooth, proper, representably smoothable morphism of derived scalloped stacks which is relatively $E$-oriented and such that $L_f$ has constant virtual rank 0. Let $i\colon \calZ\hookrightarrow \calY$ be a closed derived substack of $\calY$, and consider the pullback 
    \begin{equation}\label{diag:oriented_base_change}
        \begin{tikzcd} 
        \calX_\calZ\arrow{r}{i_{\calX_\calZ}} \arrow[swap, "f'"]{d}& \calX \arrow{d}{f}\\ 
        \calZ \arrow{r}{i_\calZ} & \calY
        \end{tikzcd}
    \end{equation}
    where $f' = i_\calZ^*f$. If $f$ is oriented by $\rho_f$, then $f'$ is oriented by $i_\calZ^*\rho_f$. Furthermore, $f_!^{\mathrm{or}}i_\calZ^* \simeq i_{\calX_\calZ}^*(f')_!^{\mathrm{or}}$. 
\end{proposition}

\subsection{Equivariant degree and local degree} \label{subsection:degree_for_schemes}

 We consider the setup of Subsection \ref{subsection:degree_for_scalloped} for $G$-equivariant maps between $G$-schemes over $\Spec k$ for finite $G$ and $k$ a field. Let $G$ be a finite group scheme that is tame over $\Spec (k)$ \cite[Definition 2.26, Example 2.27]{Hoyois6ff}. For example, this includes the case when $G$ is a finite abstract group and $|G|$ is coprime to $\op{char}(k)$ and the case when $G$ is linearly reductive over $k$. We could also $G$ be a nice group in the sense of \cite[Definition 2.1]{KR-generalized-cohomology} and formulate the $G$-degree and local degree under different assumptions, but we focus on the case where $G$ is finite. When $G$ is finite and tame, the six functor formalisms of \cite{KR-generalized-cohomology} used in Section \ref{subsection:degree_for_scalloped} for scalloped stacks agrees with that of Hoyois for $G$-quasi-projective schemes \cite{Hoyois6ff}.

Let $E_G\in \SH^G(k)$ be an equivariant motivic ring spectrum. Let $X$ and $Y$ be $G$-schemes of the same dimension $d$, and let $f\colon X\to Y$ be a proper, quasi-smooth 
$G$-equivariant, equivariantly smoothable morphism. By \cite[Tag 068E Lemma 37.62.11]{stacks-project}, $f$ is a local complete intersection morphism. 
Let  $L_f\in K_0^G(X)$ denote the class of the relative cotangent complex of $f$, 
note that $L_f$ is perfect and concentrated in cohomological degrees $[-1,0]$. 
Assume $f$ is relatively $E_G$-oriented as in Definition \ref{defn:rel_oriented}. Let $$\rho_f\colon E_G(X, L_f)\simeq E_G(X)$$ denote the relative orientation of $f$. Note that for every subgroup $H$ of $G$, $f$ is $H$-equivariant and is automatically $E_H$ oriented, for $E_H$  the restriction of $E_G$ to $\SH^H(k)$.

Recall from Definition \ref{defn:global_degree} we have a $G$-equivariant (proper) pushforward 
\[
f_!\colon E_G(X,L_f)\to E_G(Y).
\]
We thus have an oriented pushforward 
\[
f^{\mathrm{or}}_!\colon E_G(X)\to E_G(Y)
\]
after composing with $\rho_f\colon E_G(X)\simeq E_G(X,L_f)$. Definition \ref{defn:global_degree} specifies in the case of schemes with a $G$-action a \emph{$G$-degree}. 

\begin{definition}\label{defn:G-degree} 
The \emph{$G$-degree} of $f\colon X\to Y$ is $f^{\mathrm{or}}_!(1_X)\in E^0_G(Y)$, denoted $\deg^G(f)$.
\end{definition}

We will start with a $\Spec k(y)$-valued point $y$ in $Y$ whose orbit $G \cdot y = G \times_{G_y} \Spec k(y)$ is closed in $Y$, where $G_y$ denotes the stabilizer of $y$. We define
a local equivariant degree for the points in the fiber over $y$.

Henceforth, we use the conventional notation $B_{L}H$ for $[\Spec L/H]$, where $L/k$ is a finite field extension and 
$H$ is a finite group scheme.
We write $BG=[\Spec k/G]$ when $k$ is the base field.


Let
\[
X_y:=X \times^\mathbf{R}_Y \Spec k(y)
\]
denote the derived fiber, regarded naturally as a $G_y$-scheme. Assume
that its classical truncation $\pi_0(X_y)$ is zero-dimensional. Since
$f$ is proper and locally of finite presentation, $\pi_0(X_y)$ is then
finite over $\Spec k(y)$.

For every point $x\in |\pi_0(X_y)|$,
let $X_{y,x}\subseteq X_y$
denote the open and closed derived subscheme whose 
underlying classical topological space is $\{x\}$. 
Thus $\pi_0(X_{y,x})$ is the local artinian component of 
$\pi_0(X_y)$ supported at $x$, with residue field
$k(x)$. In particular, $\pi_0(X_{y,x})$ need not be reduced, and
$X_{y,x}$ may also carry nontrivial derived structure.

Let $G_x\subseteq G_y$ denote the stabilizer of $x$. The derived
subscheme $X_{y,x}$ is naturally a $G_x$-scheme. 
For a $G_y$-orbit $G_y\cdot x\subseteq|\pi_0(X_y)|$, set
$X_{y,G_y\cdot x} := \coprod_{x'\in G_y\cdot x}X_{y,x'}$.
Then $X_{y,G_y\cdot x}$ is a $G_y$-invariant open and closed derived
subscheme of $X_y$, and choosing $x$ in the orbit gives an equivalence
$[X_{y,G_y\cdot x}/G_y] \simeq [X_{y,x}/G_x]$.
Consequently,
\[
[X_y/G_y]
\simeq
\coprod_{\{G_y\cdot x\subseteq|\pi_0(X_y)|\}}
[X_{y,x}/G_x].
\]

When $X_y$ is classical and \'etale over $\Spec k(y)$, each
$X_{y,x}$ is simply $\Spec k(x)$, where $k(x)/k(y)$ is a finite
separable field extension. 
This holds, for instance, if $X \to Y$ is smooth
at all points lying over $y$ and the fiber is zero-dimensional.

As in the general case of Subsection \ref{subsection:degree_for_scalloped}, 
there is a homotopy cartesian diagram
\begin{equation}
    \begin{tikzcd} 
    \left[X_y/G_y\right]\arrow[dr, phantom, "\lrcorner", very near start] \arrow[swap, d,"f' := f|_{X_y}"] \arrow[r] & \left[X/G\right]\arrow[d, "f"] \\ 
    B_{k(y)}G_y \arrow[r] & \left[Y/G\right]
    \end{tikzcd} 
\end{equation}
For any $x\in |\pi_0(X_y)|$, we can restrict $f'$ to the open and closed
component $X_{y,x}$ to obtain  a diagram
\begin{equation}\label{diag:Gx_orbit}
    \begin{tikzcd} 
    \left[X_{y,x}/G_x\right]\arrow[swap, d,"f_x"] \arrow[r, "i_x"] & \left[X/G\right]\arrow[d, "f"] \\ 
    B_{k(y)}G_x \arrow[r, "i_x"] & \left[Y/G\right]
    \end{tikzcd} 
\end{equation}
which is not cartesian in general. 

The fiber $[X_y/G_y]$ is closed in $[X/G]$,  but any 
$X_{y,x}$ may not be $G_y$-invariant in $X_y$. Said differently, $X_{y,x}$ is only a $G_x$-scheme, not a $G_y$-scheme. In particular, to eventually obtain a sum of local degrees in the $G_y$-equivariant cohomology of $B_{k(y)}G_y$, we require a transfer which accounts for change of group from $G_x$ to $G_y$. We will also use a transfer which accounts for change of field for finite separable extensions in later sections. We define both now. 

\begin{definition}[Change of groups]
\label{defn:group_transfer}
Let $H$ be a subgroup of $G$ and $L/k$ a finite, separable extension, and let $B\iota \colon B_LH \to B_LG$ be the induced map. The \emph{transfer from $H$ to $G$} is the proper pushforward in $E$-cohomology along $B\iota$, 
\[
B\iota_!\colon E^0(B_LH )\to E^0(B_LG),
\]
which will be denoted $\Tr_H^G$. In scheme notation, $\Tr_H^G \colon E^0_H(\Spec L)\to E^0_{G}(\Spec L)$. 
\end{definition}

 Note transfer from $H$ to $G$ leaves the field of definition $L$ unchanged. Of course, we will be particularly interested in the transfer $\Tr_{G_x}^{G_y}$ associated with a $G_y$-orbit in $|\pi_0(X_y)|$.

\begin{definition}[Change of fields]
\label{defn:field_transfer}
     Let $k\hookrightarrow L$ be a finite, separable extension of fields with $\gamma\colon \Spec L\to \Spec k$ an $H$-equivariant map for some subgroup $H$ of $G$. The \emph{transfer from $L$ to $k$} is the proper pushforward in $E_H$-cohomology along $\gamma$, 
    \[
    \gamma_!\colon E^0(B_LH)\to E^0(B_kH),
    \]
    which will be denoted $\Tr_{L/k}$. In scheme notation, $\Tr_{L/k}: E^0_H(\Spec L)\to E^0_H(\Spec k)$. 
\end{definition}

Composing $\Tr_{L/k}$ and $\Tr_{H}^{G}$, we obtain a change of group and field, 
\begin{equation}\label{eq:group_and_field_transfer}
E^0(B_L H)\stackrel{\Tr_{L/k}}{\longrightarrow}E^0(B_{k} H)\stackrel{\Tr_{H}^{G}}{\longrightarrow} E^0(B_k G). 
\end{equation}
The field transfer will not reappear in this section, but becomes important in Section \ref{section:bezout}. 

\begin{exa}\label{example:transfer} Let $L$ be a finite extension of field $k$. 
Let $E$ be the motivic spectrum representing homotopy invariant K-theory, 
$\mathrm{KGL}$, and for a finite group scheme $G$ write $R_k(G)$ (resp. $R_L(G)$) for the representation ring of $G$ over $k$ (resp. $L$).
Let $H$ be a subgroup of $G$ and write 
$B\iota \colon B_LH \to B_LG$ for the induced map.
The transfer from $H$ to $G$ is  
\[
B\iota_!\colon R_L(H) \cong \mathrm{KGL}^0(B_LH) \to \mathrm{KGL}^0(B_LG) \cong R_L(G)
\]
defined by $V\mapsto \ind_H^G V := V\otimes L[G/H]$. When $BG = [\Spec(k) / G]$ and $B_LH = [\Spec (L) / H]$, the composition of $\Tr_{L/k}$ and $\Tr_H^G$ from \eqref{eq:group_and_field_transfer} is the map 
\[
\Tr_{L/k} \Tr_H^G\colon R_L(H)\cong \KGL^0(B_LH)\to \KGL^0(BG) \cong R_k(G)
\]
given by
\[
V\mapsto \op{Res}_{L/k}(\ind_H^G V) 
\] 
for $V\in R_L(H)$. Moreover if $V$ has trivial $H$-action, then 
\[
\op{Res}_{L/k}(\ind_H^G V) \cong \op{Res}_{L/k}(V)\otimes k[G/H].
\]

Suppose $A$ is a connective quasi-smooth $k$-algebra with $\pi_0(A)=L$ a finite field extension of $k$. 
For example, take $L$ to be a finite separable extension of $k$ and $\Spec A$ to be the derived self intersection of the origin inside $\mathbb{A}^1_L$, i.e., $A = L \otimes^{\mathbf{L}}_{L[t]} L$ where both maps are given by evaluation at $t=0$.
Consider $V \in R_L(H)$ as above and its corresponding class in 
\[
\mathrm{KGL}^0(B_A H) \cong \mathrm{KGL}^0(B_L H) \cong R_L(H),
\]
where the first isomorphism holds by derived invariance.
The pushforward of the class in $\mathrm{KGL}^0(B_k G)$ along the quasi-smooth map $B_A H \to B_k G$ is given by
\[
\left(\sum_{j\geq 0}(-1)^j \dim_L\!\bigl(\wedge^j H^{-1}(L_{A/L})\bigr)\right)
\cdot
\bigl[k[G]\otimes_{k[H]}\Res_{L/k}(V)\bigr].
\]
Since $L_{A/L}$ has Tor-amplitude in $[-1,-1]$, this coefficient is
\[
\sum_{j\geq 0}(-1)^j \dim_L\!\bigl(\wedge^j H^{-1}(L_{A/L})\bigr)
=
(1-1)^{\dim_L H^{-1}(L_{A/L})}.
\]
Hence the class is $0$ unless $H^{-1}(L_{A/L})=0$, equivalently it is zero unless $A=L$ is a finite separable field extension of $k$ as above.

\end{exa}

We now define the local equivariant degree. Since
$[X_{y,x}/G_x]$ is an open and closed substack of the derived base
change $[X_y/G_y]$, the relative orientation $\rho_f$ induces by
base change and restriction a relative $E_{G_x}$-orientation
\[
\rho_{f_x}:=i_x^*\rho_f\colon E_{G_x}(X_{y,x},L_{f_x}) \simeq E_{G_x}(X_{y,x})
\]
of $f_x\colon [X_{y,x}/G_x]\to B_{k(y)}G_x$.

\begin{definition}\label{defn:local_G_degree}
 Let 
    \[
    f_x\colon [X_{y,x}/G_x] \to B_{k(y)}G_x
    \]
    be the restriction of the base change of $f\colon X\to Y$ to to the open and closed component $X_{y,x}\subseteq X_y$, which is only a $G_x$-equivariant map of derived schemes. 
    As in the global case, let 
    $$(f_{x})_{!}\colon 
    E^0_{G_x}(X_{y,x},L_{f_x}) \to E^0_{G_x}(\Spec k(y))$$ 
    denote the proper pushforward. Composing with the relative orientation $\rho_{f_x}$, we have an  oriented pushforward in degree 0,
\[
 E^0_{G_x}(X_{y,x})\stackrel{\rho_{f_x}}{\longrightarrow} E^0_{G_x}(X_{y,x}, L_{f_x})\stackrel{(f_{x})_{!}}{\longrightarrow} E^0_{G_x}(\Spec k(y)),
\]
which we will denote by $(f_{x})_{!}^{\mathrm{or}} \colon E^0_{G_x}(X_{y,x}) \to E^0_{G_x}(\Spec k(y)).$ The \emph{local equivariant degree} of $f$ at $x$ is $\deg^{G_x}_xf:=(f_{x})_{!}^{\mathrm{or}}(1_{X_{y,x}})$ in $E^0_{G_x}(\Spec k(y))$. 
\end{definition}

\begin{remark}
    When $X_{y,x}\simeq \Spec k(x)$ and $k(x)/k(y)$ is finite separable, 
    the map $$f_x\colon B_{k(x)}G_x\to B_{k(y)}G_x$$
    is \'etale and hence is canonically relatively oriented.
    However, this orientation need not agree with the orientation induced from $f$.
    In Definition \ref{defn:local_G_degree}, we therefore require the local degree at $x$ to be defined using the relative $E_{G_x}$-orientation induced from $\rho_f$ so that the relative orientation of $f_x$ is compatible with that of $f$. 
\end{remark}

We now state and prove the local to global degree result for $G$-schemes.

\begin{theorem}\label{thm:local_global_for_schemes} Let $E_G \in \SH^G(k)$ be a motivic ring spectrum, and let $f\colon X\to Y$ be a proper, quasi-smooth, equivariantly smoothable $G$-equivariant map between $G$-schemes $X$ and $Y$ of the same dimension. Assume $f$ is relatively $E_G$-oriented.  
Let $i_y\colon \Spec k(y)\rightarrow Y$ be a point of $Y$ whose $G$-orbit is closed in $Y$ and suppose that the
classical truncation of the derived fiber $X_y=X \times^\mathbf{R}_Y \Spec k(y)$
is zero dimensional.
Then 
\begin{equation}\label{eq:G_deg_eq}
i_y^* \deg^{G_y}f = \sum_{\{G_y\cdot x\subseteq|\pi_0(X_y)|\}}
\Tr_{G_x}^{G_y}\deg_x^{G_x} f
\end{equation}
in $E^0_{G_y}(\Spec k(y))$.
\end{theorem} 

\begin{proof}
    Since $f$ is proper and locally of finite presentation and
    $\pi_0(X_y)$ is zero-dimensional, $\pi_0(X_y)$ is finite over $k(y)$.
    Thus $|\pi_0(X_y)|$ is a finite set.
    
    For each $G_y$-orbit $G_y\cdot x$ in $|\pi_0(X_y)|$, let
    $X_{y,G_y\cdot x} = \coprod_{x'\in G_y\cdot x}X_{y,x'}$. Then
    \[
    [X_y/G_y] \simeq \bigsqcup_{\{G_y\cdot x\subseteq |\pi_0(X_y)|\}} [X_{y,x}/G_x].
    \]
     
     For each such orbit, we have a commutative diagram
     For each such orbit, we have a commutative diagram
    \begin{equation}\label{diag:big_G_local_global}
    \begin{tikzcd}
    \left[X_{y,x}/G_x\right] \ar[r] \ar[d,"f_x"'] &
    \left[X_y/G_y\right] \ar[r,"i_{X_y}"] \ar[d,"f'"'] &
    \left[X/G\right] \ar[d,"f"]
    \\
    B_{k(y)}G_x \ar[r,"B\iota"'] &
    B_{k(y)}G_y \ar[r,"i_y"'] &
    \left[Y/G\right],
    \end{tikzcd}
    \end{equation}
    where $f'$ is the pullback of $f$ to $X_y$, and $B\iota \colon B_{k(y)}G_x\to B_{k(y)}G_y$ over $\Spec k(y)$ is the map induced by inclusion of groups $G_x \subseteq G_y$ whose proper pushforward is the group transfer $\Tr_{G_x}^{G_y}$ on cohomology that leaves the field of definition $k(y)$ unchanged.  
    The right square is cartesian. On the level of underlying schemes, the right square is $G_y$-equivariant, but the left square and the outside square are only $G_x$-equivariant.

    Following the proof of Proposition \ref{prop:local_to_global},
    \begin{align}\label{eq:big_eqn_in_main_proof}
    i_y^*\deg^{G_y}(f) & := i_{y}^*f_!^{\mathrm{or}}(1_{[X/G]}) \notag \\ 
    &= (f')_!^{\mathrm{or}}(i_{X_y})^*(1_{[X/G]})\notag \\
    &= (f'_!)^{\mathrm{or}}(1_{[X_y/G_y]}) \notag\\
    &= \sum_{\{G_y\cdot x\subseteq|\pi_0(X_y)|\}} B\iota_!
     (f_x)_!^{\mathrm{or}}(1_{X_{y,x}}) \notag\\
    &=\sum_{G_y \cdot x \subseteq X_y} \Tr_{G_x}^{G_y}(f_{x})_!^{\mathrm{or}}(1_{X_{y,x}}) , 
    \end{align}
where the last equality holds by compatibility of the proper pushforward $(f_x)_!^{\mathrm{or}}$ with the transfer $\Tr_{G_x}^{G_y}$, i.e., the proper pushforward in cohomology along $B\iota$. Observe that $$(f_{x})_!^{\mathrm{or}}(1_{X_{y,x}}) \in E^0_{G_x}(k(y)),$$ so $\Tr_{G_x}^{G_y}(f_{x})_!^{\mathrm{or}}(1_{X_{y,x}}) \in E^0_{G_y}(k(y))$ for all $x$. 
 By definition, \eqref{eq:big_eqn_in_main_proof} is equal to 
 \[
\sum_{\{G_y\cdot x\subseteq |\pi_0(X_y)|\}} \Tr_{G_x}^{G_y}\deg^{G_x} f_x= \sum_{\{G_y\cdot x\subseteq |\pi_0(X_y)|\}} \Tr_{G_x}^{G_y} \deg^{G_x}_x f,
 \]
 finishing the proof.

\end{proof}

\section{Local Multiplicities and Equivariant B\'{e}zout}\label{section:bezout}

\subsection{Euler number}\label{subsection:euler_number}

In this section, we define the Euler number of a vector bundle on a scalloped stack and show it is a degree. One can view this as the stack theoretic version of \cite[5.4]{bachmann_1-euler_2021}, and is certainly inspired by the constructions of Bachmann--Wickelgren and D\'eglise--Jin--Khan \cite{bachmann_1-euler_2021,DJK}. We will specialize to the equivariant case in Subsection \ref{subsection:G_euler_number}. 
 
We start by introducing the definition of a relatively oriented vector bundle. 

\begin{definition}\label{defn:rel_oriented_bdle}
Let $E$ be a motivic ring spectrum over a scalloped stack $\calS$. Let $\pi_X: \calX \to \calS$ be a representable map over $\calS$.
We will say that a vector bundle $p: \calV = \mathbb{V}(\calE) \to \calX$ is \emph{relatively $E$-oriented}, or simply \emph{$E$-oriented}, if there exists an equivalence $$E_\calV \otimes \Th_\calV(p^*\calE) \xrightarrow{\simeq} E_\calV \otimes \Th_\calV(p^*L_{\pi_{\calX}}).$$ 
\end{definition} 

Definition \ref{defn:rel_oriented_bdle} is analogous to \cite[Defn.~4.15]{brazeuler}.  By homotopy invariance the above definition is equivalent to saying that there is an equivalence 
    \begin{equation} \label{eq:VB orientation}
      \rho_\calE \colon E_\calX \otimes \Th_\calX (\calE) \xrightarrow{\simeq} E_\calX \otimes \Th_\calX (L_{\pi_{\calX}}). 
    \end{equation}
    In particular, the tangent bundle of a smooth representable scalloped stack $\calX$ over $\calS$ is always $E$-oriented for any ring spectrum $E$ over $\calS$.

Note this forces $\op{rank}\calV = \dim \calX$ over $\calS$. 

\begin{definition} \label{defn:Euler-class} 
Let $p: \calV = \mathbb{V}(\calE) \to \calX$ be a vector bundle over a scalloped stack $\calX$. Let $0: \calX \to \calV$ denote its zero section. The \emph{Euler class} of $\calV$ is defined to be the class $$e^E(\calV) = 0^* 0_!(1_\calX)$$  in $E^0(X,\calE)$.  
\end{definition}

    With this notion of orientation and Euler class, we can define the $E$-valued Euler number for relatively $E$-oriented vector bundles.

\begin{definition}\label{defn:euler_number}
Let $\pi_\calX \colon \calX \to \calS$ be a quasi-smooth, proper, representably smoothable scalloped stack over $\calS$. Let $$\calV := \mathbb{V}(\calE)  \stackrel{p}{\longrightarrow} \calX$$ be a vector bundle. 
Let $E$ be a motivic ring spectrum over $\calS$. 
Suppose $\calV$ is an $E$-oriented vector bundle oriented by $$\rho_\calE: E_\calX \otimes \Th_\calX (\calE) \xrightarrow{\simeq} E_\calX \otimes \Th_\calX (L_{\pi_{\calX}}).$$ We obtain a composite map on cohomology groups
\begin{equation}\label{eq:euler-num-in-right-degree}
    \begin{tikzcd} 
     E^0(\calX, \calE)  \xrightarrow[\simeq]{\rho_{\calE}} E^0(\calX, L_{\pi_{\calX}})  
     \arrow["\pi_{\calX,!}"]{r} & E^0(\calS). \\
    \end{tikzcd}
\end{equation}
Given any section $\sigma: \calX \to \calV$ of $\calV$, $\sigma^*0_!(1_\calX)\in E^0(\calX, \calE)$. The \emph{Euler number} of an oriented vector bundle $\calV$ with respect to a section $\sigma$ is 
    \begin{equation} \label{eq:euler-number-defn}
    n^E(\calV, \sigma) := \pi_{\calX,!} \rho_\calE (\sigma^*0_!(1_{\calX}))
    \end{equation}
    in $E^0(\calS)$. 
    Note that when $\sigma $ is chosen to be the zero section \eqref{eq:euler-number-defn} becomes $\pi_{\calX,!} \rho_\calE  (e^E(\calV))$ . 
\end{definition}

    Given any section $\sigma: \calX \to \calV$ of a vector bundle $\calV\to \calX$, let $i: Z(\sigma) \hookrightarrow \calX$ denote the inclusion of the derived zero locus $Z(\sigma)$ of $\sigma$ in $\calX$ and let $\pi_{\calZ}: Z(\sigma) \to \calS$ denote the structure map, note $\pi_{\calZ} = \pi_{\calX} \circ i$. 
    If the vector bundle $\calV = \mathbb{V}(\calE)$ is $E$-oriented, the map $\pi_\calZ$ has a canonical relative $E$-orientation given by the equivalence:
    \begin{equation} \label{eq:0locus-orientation}
    E_{Z(\sigma)} \xrightarrow[\simeq]{i^*(\rho_\calE \langle -\calE \rangle)} 
    E_{Z(\sigma)} \otimes \Th_{Z(\sigma)}(i^* L_{\calX/\calS}) \otimes \Th_{Z(\sigma)}(-i^* \calE) \simeq
    E_{Z(\sigma)} \otimes \Th_{Z(\sigma)}(L_{\pi_\calZ}),
    \end{equation}
    where the first equivalence is induced by the $E$-orientation of $i^*\calV$ over $Z(\sigma)$, see \eqref{eq:VB orientation} and last equivalence is induced by the fiber sequence of cotangent complexes 
    $$i^*L_{\calX/\calS} \to L_{Z(\sigma)/\calS} \to L_{i} \simeq i^*\calE[1].$$ 
    Let $\rho_{\calZ} := i^*(\rho_\calE \langle -\calE \rangle)$ denote this induced relative orientation of $\pi_{\calZ}$.
    
    We show Definition \ref{defn:euler_number} is independent of choice of section, as is desired from any Euler number: 

\begin{proposition}[Independence of section] \label{prop:indep.-of-sec}
    Given an $E$-oriented vector bundle $p: \calV \to \calX$ on a quasi-smooth, proper, representably smoothable stack $\pi_\calX \colon \calX \to \calS$ over a scalloped stack $\calS$, the Euler number $n^E(\calV,\sigma)$ is independent of the choice of the section $\sigma\colon \calX\to \calV$.
\end{proposition}

\begin{proof}
    Let $\calV = \mathbb{V}(\calE)$ and let $\rho_\calE$ denote the orientation of the vector bundle $\calV$.
    By homotopy invariance, $$p^*: E^0(\calX) \to E^0(\calV)$$ is an equivalence. 
    Since $(p0)^* = (p\sigma)^*$ is the identity, homotopy invariance implies
    \[
    \sigma^* = 0^*: E^0(\calX) \to E^0(\calV).
    \]
    Hence $n^E(\calV,\sigma)$ becomes $\pi_{\calX,!}\rho_\calE (e^E(\calV))$
    for any section  $\sigma\colon \calX\to \calV$. This shows $n^E(\calV,\sigma)$ is independent of the choice of $\sigma$.  
\end{proof}

    Finally, we show that the Euler number of a vector bundle $\calV$ with respect to a section $\sigma$ agrees with the usual geometric notion that $n^E(\calV, \sigma)$ should give an oriented count of zeros of $\sigma$, given by the degree of $Z(\sigma)$.

    \begin{lemma}\label{lem:equiv-defns-of-euler-number}
        Let $E$ be a motivic ring spectrum over a scalloped stack $\calS$. Assume that $\pi_{\calX}\colon \calX \to \calS$ is a proper, quasi-smooth, representably smoothable stack over $\calS$. Assume $p\colon \calV\to \calX$ is an $E$-oriented vector bundle over $\calX$, 
         and let $\sigma$ be a section of $\calV$. 
    Then 
    \[
    n^E(\calV,\sigma)=\deg(Z(\sigma)) 
    \] in $E^0(\calS)$, where $\deg(Z(\sigma)) :=(\pi_\calZ^{\mathrm{or}})_!(1_{Z(\sigma)})$ and $\pi_\calZ: Z(\sigma) \to \calS$ has the induced relative orientation. 
    \end{lemma}

    \begin{proof}
    Let $\calE$ denote the locally free sheaf on $\calX$ defining the vector bundle $\calV$, i.e., $\calV = \mathbb{V}(\calE)$ and let $i: Z(\sigma) \to \calX$ denote the inclusion of the derived locus. 
    Let $\rho_\calE$ denote the orientation of the vector bundle $\calV$ and $\rho_\calZ:= i^*(\rho_\calE \langle -\calE \rangle)$ denote the induced relative orientation on $Z(\sigma)$ as in \eqref{eq:0locus-orientation}. 

    Consider the diagram below, where the right square is homotopy cartesian 
    \begin{equation}\label{diag:section-independence}
    \begin{tikzcd}
        & Z(\sigma) \arrow[dr, phantom, "\lrcorner", very near start] \arrow["i"]{r} \arrow[swap,"i"]{d} \arrow[swap,  "\pi_{\mathcal{Z}}"]{dl}& \mathcal{X} \arrow{d}{0} \\
         \mathcal{S} & \mathcal{X} \arrow{l}{\pi_{\calX}} \arrow{r}{\sigma} & \mathcal{V} ,
    \end{tikzcd}
    \end{equation}
    and $\pi_\calZ = \pi_\calX \circ i$.

By base change formula for the above square, we get $\sigma^*0_!(1_\calX) = i_!i^*(1_\calX) = i_!(1_{Z(\sigma)})$.
Combined with \eqref{eq:euler-number-defn}, 
    $$
    n^E(\calV,\sigma) := \pi_{\calX,!} \rho_\calE (\sigma^*0_!(1_{\calX})) = \pi_{\calX,!} \rho_\calE (i_!(1_{Z(\sigma)})),
    $$
Recall from \eqref{eq:oriented_pushforward}
    $$
    \deg(Z(\sigma)) :=(\pi_\calZ^{\mathrm{or}})_!(1_{Z(\sigma)}) = \pi_{\calZ,!} \rho_\calZ (1_{Z(\sigma)}) = \pi_{\calX,!} i_! \rho_\calZ (1_{Z(\sigma)}).
    $$
 Therefore, the required equality is reduced to proving
 \[
 \rho_\calE i_! = i_! \rho_\calZ = i_! (i^* \rho_\calE \langle -\calE \rangle).
 \]
This holds by the construction of Gysin pushforwards on cohomology using the six operations (see \eqref{eq:Gysin-push}),  the compatibility of Thom twists with the operations, and the projection formula for $i_!$.
 \end{proof}

 We denote the Euler number henceforth by $n^E(\calV)$ without reference to a section $\sigma$ in light of Proposition \ref{prop:indep.-of-sec}.

In the case when $X$ is a $G$-scheme, $V$ is an equivariant bundle on $X$, and $\sigma\colon X\to V$ is an equivariant section, Definition \ref{defn:euler_number} does not require the section $\sigma$ to have simple \emph{and} isolated zeros. This is discussed further in Subsection \ref{subsection:G_euler_number}. See Definition \ref{defn:isolated_simple_zeros} for the definition of simple and isolated zeros in the case of schemes. 

\subsubsection{Intersection product}
Given a morphism $f:\calX \to \calY$, 
we start by defining twisted versions of relative orientations (see Definition \ref{defn:rel_oriented}), which allows us to define oriented pushforward between twisted cohomology groups.

\begin{definition}\label{defn:twisted_rel_oriented}
Let $\alpha \in K(\calX)$.
An \emph{$\alpha$-twisted relative $E$-orientation} of $f: \calX \to \calY$ is an equivalence of motivic spectra
\[
\rho^\alpha_f\colon E_\calX \otimes \Th_\calX(\alpha) \stackrel{\sim}{\longrightarrow} E_\calX \otimes \Th_\calX(L_f).
\]
Note that for $\alpha = 0$, this recovers the notion of relative $E$-orientation from Definition \ref{defn:twisted_rel_oriented}.
An $\alpha$-twisted relative $E$-orientation as above induces an equivalence 
\[
\rho^\alpha_f\colon E^0(\calX, \alpha) \stackrel{\sim}{\longrightarrow} E^0(\calX,L_f),
\]
which we also denote by $\rho^\alpha_f$ as before by abuse of notation. 

An $\alpha$-twisted relative $E$-orientation of $f$ defines pushforwards $f_!^{\mathrm{or}, \alpha}:E^0(\calX, \alpha)\to E^0(\calY)$ and
$f_!^{\mathrm{or}, \alpha}:E^0(\calX)\to E^0(\calY, - \alpha)$ as the compositions
\[
E^0(\calX,\alpha) \xrightarrow{\rho^\alpha_f} E^0(\calX,L_f) \xrightarrow{f_!} E^0(Y) \quad
\mathrm{,} \quad
E^0(\calX) \xrightarrow{\rho^\alpha_f} E^0(\calX,L_f-\alpha) \xrightarrow{f_!} E^0(Y, -\alpha).
\]
\end{definition}

Let $\calX$ be a derived scalloped stack representable over a scalloped stack $\calS$. Let $i_1: \calZ_1 \to \calX$, $i_2: \calZ_2 \to \calX$ be quasi-smooth, closed substacks of $\calX$. 
Let $i: \calZ := \calZ_1 \times^\mathbf{R}_\calX \calZ_2 \to \calX$ denote the derived intersection of $\calZ_1$ and $\calZ_2$ in $\calX$, 
\begin{center}
    \begin{tikzcd}
        \calZ := \calZ_1 \times^\mathbf{R}_\calX \calZ_2 \arrow{r}\arrow{d} \arrow["i", dashed]{dr} & \calZ_1 \arrow["i_1"]{d}\\
        \calZ_2 \arrow[swap, "i_2"]{r}& \calX. 
    \end{tikzcd}
\end{center}
For $j=1,2$, suppose given $e_j \in K(X)$ and $(i_j^*e_j)$-twisted relative $E$-orientations of $i_j$:
\[
\rho_j \colon E_{\calZ_j} \otimes \Th_{\calZ_j}(i_j^*(e_j)) \stackrel{\sim}{\longrightarrow} E_{\calZ_j} \otimes \Th_{\calZ_j}(L_{i_j}),
\]
and an $(-e_1-e_2)$-twisted $E$-orientation of $\calX \to \calS$:
\[
\theta \colon E_\calX \otimes \Th_\calX(-e_1-e_2) \stackrel{\sim}{\longrightarrow} E_\calX \otimes \Th_\calX(L_\calX).
\]
This data defines a $i^*(e_1+e_2)$-twisted relative $E$-orientation $\rho'$ of $i \colon \calZ \to \calX$ and an (untwisted) relative $E$-orientation $\rho$ of the map $\pi_\calZ \colon \calZ \to \calS$:
\[
\rho'\colon E_\calZ \stackrel{\rho_1}{\longrightarrow} E_\calZ \otimes \Th_\calZ(-i^*e_1 + L_{i_1}\mid_\calZ) \stackrel{\rho_2}{\longrightarrow} E_\calZ \otimes \Th_\calZ(-i^*e_1 - i^*e_2 + L_i), 
\]
\[
\mathrm{and} \, \,
\rho := \theta \circ \rho' \colon E_\calZ \stackrel{\rho'}{\longrightarrow} E_\calZ \otimes \Th_\calZ(-i^*e_1 - i^*e_2 + L_i) \stackrel{\theta}{\longrightarrow} E_\calZ \otimes \Th_\calZ(L_\calZ),
\]
where we use the equivalences, $\Th_\calZ(L_{i_1}\mid_\calZ + L_{i_2}\mid_\calZ) \simeq \Th_\calZ(L_i)$ and $\Th_\calZ(L_i + i^*L_\calX) \simeq \Th_\calZ(L_\calZ)$.
The $E$-intersection product of $\calZ_1$ and $\calZ_2$ in $\calX$ with respect to the orientation data $\rho'$ is defined as the $E$-cohomology class $i_!^{\mathrm{or},E_1+e_2}(1_{\calZ}) \in E^0(\calX, e_1+e_2)$.
It follows from the proof of \cite[Thm.~3.22]{khan-VFC} that 
\[
i_!^{\mathrm{or}, e_1+e_2}(1_{\calZ}) = i_{1,!}^{\mathrm{or},e_1}(1_{\calZ_1}) \cdot i_{2,!}^{\mathrm{or},e_2}(1_{\calZ_2}).
\]

The \emph{$E$-intersection number} of $\calZ_1$ and $\calZ_2$ in $\calX$ is defined as the $E$-cohomology class 
\[
\pi_{\calZ, !}^{\mathrm{or}}(1_{\calZ}) \in E^0(\calS),
\]
where we take the (untwisted) oriented pushforward with respect to the orientation $\rho$ of $\pi_\calZ$.

Note that this coincides with the degree of the intersection product class, and therefore recovers the definition of intersection number in the usual sense whenever both are defined and we work with canonical, untwisted, orientations.

For $e_1 = e_2 = 0$, the data above correspond to relative $E$-orientations of the maps $\calZ_1 \xrightarrow{i_1} \calX$, $\calZ_2 \xrightarrow{i_2} \calX$ and $\calX \to \calS$. 
Suppose $i_1 \colon \calZ_1 \to \calX$ and $i_2\colon \calZ_2 \to \calX$ are given by the vanishing locus of sections $s_1, s_2$ of vector bundles $\mathbb{V}(\calE_1), \mathbb{V}(\calE_2)$, respectively, over a proper, quasi-smooth, representably smoothable, derived stack $\calX$ over the scalloped stack $\calS$. 
Then the derived intersection $\calZ$ is the derived vanishing locus of the section $(s_1, s_2)$ of the vector bundle $V:=\mathbb{V}(\calE_1 \oplus \calE_2)$ over $\calX$. 
For $j=1,2$, let $e_j = -[\calE_j]$, then $L_{i_j} \simeq \calE_j [1]$ implies $i_j$ is canonically $e_j$-twisted $E$-oriented. 
Now assume $V$ is oriented on $\calX$ over $\calS$.
Then this defines an $(-e_1-e_2)$-twisted $E$-orientation of $\calX \to \calS$. Therefore the $E$-intersection number of $\calZ_1$ and $\calZ_2$ in $\calX$, whence the degree of the intersection $\calZ$, is defined and coincides with the Euler number 
\begin{equation} \label{eq:int-no.=Euler-no.}
    n^E(V),
\end{equation}
which is independent of the choice of sections $s_1$ and $s_2$ by Proposition~\ref{prop:indep.-of-sec}.


\subsection{Equivariant Euler number and B\'{e}zout's Theorem}\label{subsection:G_euler_number} 

In this section, we assume $S$ is the spectrum of a field $k$ and work relative to the base $\calS = BG$ (over $S$), the classifying stack of a finite abstract group $G$ of order invertible in $k$. Let $E_G \in \SH^G(S) = \SH(\calS)$ be a motivic ring spectrum. We now specialize the non-transverse Euler number of Subsection \ref{subsection:euler_number} to the $G$-equivariant case. When we assume a section has only \emph{simple} isolated zeros and $E$ is an oriented ring spectrum, the equivariant Euler number has a formula in terms of transfers of $1$. See Proposition \ref{prop:Euler-number=local_deg}. This is analogous to \cite[Lemma 5.21]{brazeuler}. When a section has isolated zeros that are not necessarily simple, the Euler number is still a sum of transfers of local degrees, see Theorem \ref{thm:non_transverse_G_euler=local_degrees}. 

First, recall the statement of ordinary B\'ezout's theorem over a closed field of characteristic 0, which can be found for example in \cite[[Corollary 2.4]{EH16}: 

\begin{theorem}[B\'ezout's Theorem] If $X_1, \ldots, X_k\subset \mathbb{P}^n$ are subvarieties of codimensions $c_1, \ldots, c_k$, with $\sum c_i\leq n$, and the $X_i$ intersect generically transversally, then 
\[
\deg(X_1\cap\cdots\cap X_k) = \prod \deg(X_i). 
\]
In particular, two subvarieites $X,Y\subset \mathbb{P}^n$ having complimentary codimension and intersecting transversely will intersect in exactly $\deg(X)\cdot \deg(Y)$ points. 
\end{theorem}

 $X$ will be a $G$-equivariant proper, quasi-smooth, equivariantly smoothable algebraic space over $S$ throughout this section. Note that $[X/G]\to \calS$ is representable. For an $E_G$-oriented equivariant vector bundle $V \to X$, we denote the equivariant Euler class $e^E([V/G])$ of $V$ by $e^{E_G}(V)$ and its Euler number by $n^{E_G}(V)$.

\begin{definition} \label{defn:isolated_simple_zeros}
Recall from \cite[Defn.~22]{27lines} the definition of isolated and simple zeros of a section of a vector bundle. Given a vector bundle $p:V \to X$ with $\op{rank} V = \dim_k X$ and a section $\sigma: X \to V$ on a smooth scheme $X$. A point $x\in  \pi_0Z(\sigma)$, the underived zero locus of $\sigma$,  is said to be an {\it isolated zero} if $\calO_{Z,x}$ is a finite $k$-algebra. An isolated zero is {\it simple} if the Jacobian determinant of $\sigma$ at $x$ is non-vanishing. In this case the zero locus is étale over $k$ at $x$, and hence $k(x)/k$ is a finite separable extension. A section $\sigma$ is said to have \emph{isolated and simple zeros} if it has finitely many zeros and all its zeros are isolated and simple. 
This is equivalent to saying that the underlying topological space of
the classical zero locus $\pi_0Z(\sigma)$ consists of finitely many
closed points \cite[Proposition~23]{27lines}. Note that $Z(\sigma) = \pi_0Z(\sigma)$ under these conditions. 
\end{definition}


Note that if $p: V \to X$ is an equivariant vector bundle and $\sigma$ is an equivariant section with isolated and simple zeros, then $Z(\sigma)$ can be written as a disjoint union of $G$-invariant orbits $$\coprod_{\{G\cdot x_i \subseteq Z(\sigma)\}} Gx_i,$$ where each $x_i$ is an isolated simple zero of $\sigma$.

\begin{definition}\label{defn:compatibly_oriented_zeros}
    Let $E_G\in \SH^G(S)$ be a motivic ring spectrum over $S=\Spec(k)$. Let $p\colon V\to X$ be a relatively $E_G$-oriented $G$-equivariant vector bundle on a smooth algebraic space $X$ with $G$-action over $S$ with $\op{rank}V=\dim X$. Let $\sigma$ be an equivariant section of $V\to X$ with isolated and simeple zeros. Then $Z(\sigma)\to S$ is finite \'etale, and hence $L_{Z(\sigma)/S}\simeq 0$.
    We say $\sigma\colon X\to V$ is \emph{compatibly $E_G$-oriented} if the induced orientations on the zeros $Z(\sigma)$ agrees with the canonical orientation induced by $L_{Z(\sigma)/S}\simeq 0$. 
\end{definition}

\begin{proposition} \label{prop:Euler-number=local_deg}
    Let $G$ be a finite group which is tame over $k$. Let $X$ be a proper, smooth algebraic space over $S= \Spec(k)$ 
    with $G$-action and $p: V \to X$ be an $E_G$-oriented equivariant vector bundle on $X$ with $\op{rank}V=\dim X$. Suppose $p$ has an equivariant section $\sigma$ with isolated and simple zeros $\coprod Gx_i$. 
    
    Let $E_G \in \SH^G(S)$ be any motivic ring spectrum. If $\sigma$ is compatibly $E_G$-oriented, then the equivariant Euler number with respect to $\sigma$ can be computed as the sum of local indices:
    \[ 
    n^{E_G}(V, \sigma) = \sum_{\{G\cdot x_i \subseteq Z(\sigma)\}} \Tr^G_{G_{x_i}} \Tr_{k(x_i)/k} (1_{G_{x_i}})
    \]
   in $E_G^0(S)$, where $G_{x_i} \subseteq G$ denotes the stabilizer group of $x_i \in Z(\sigma)$.
  By independence of section (Proposition~\ref{prop:indep.-of-sec}), this computes $n^{E_G}(V)$.
\end{proposition}

\begin{proof}
 It follows from the base change formula applied to the following equivariant homotopy cartesian square:
\[
    \begin{tikzcd}
        \underset{\{G\cdot x_i \subseteq Z(\sigma)\}} {\coprod} Gx_i \arrow{r} \arrow{d} & X \arrow{d}{\sigma} \\
        X \arrow{r}{0} & V.
    \end{tikzcd}
    \]
Since  $B_{k(x)}G_{x} \to BG$ coincides with the composite $$B_{k(x)}G_{x} \to BG_{x} \to BG,$$ we get
\[
n^{E_G}(V) = \sum_{\{G\cdot x_i \subseteq Z(\sigma)\}} \Tr^G_{G_{x_i}} \iota_!(1),
\]
 where $\iota: B_{k(x_i)}G_{x_i} \to BG_{x_i}$ denotes the natural map.
Now note that $\iota$ is a base change of the map $\Spec(k(x)) \to \Spec(k)$ along $BG \to \Spec(k)$. 
\end{proof}

\begin{remark}
    When a section $\sigma$ has isolated but not necessarily simple zeros, a computation of $n^{E_G}(V)$ in terms of zeros of $\sigma$ is still possible using local degrees, see Theorem \ref{thm:non_transverse_G_euler=local_degrees} below. However, the local degrees appearing in Theorem \ref{thm:non_transverse_G_euler=local_degrees} may not be $1_{G_{x_i}}\in E^0_{G_{x_i}}(B_{k(x_i)}G_{x_i})$ even when orientations are canonical. 
\end{remark}

\subsubsection{Euler number as sum of local degrees}
In this Subsection, we expand on the case when the equivariant Euler number is computed using a section with isolated zeros that are not necessarily simple. The main application is Theorem \ref{thm:bezout}, which can be thought of as an equivariant non-transverse B\'ezout's theorem. 

\begin{construction} \label{construction:map_f_for_euler}
Let $V := \mathbb{V}(\calE) \to X$ be a $G$-equivariant bundle over a proper, quasi-smooth, equivariantly smoothable derived algebraic space $X$ over $S = \Spec(k)$ with $G$-action.
Let $\sigma\in \Gamma(X,V)^G$ be a $G$-invariant section of $V$ and let $\tilde{U}\subset\Gamma(X,V)^G$ be any finite dimensional $G$-invariant subspace of sections containing $\sigma$. For example, we can take $\tilde{U}=\langle \sigma\rangle$ or $\tilde{U}=\Gamma(X,V)^G$ since $\Gamma(X,V)$ is finite dimensional \cite[Tag 08AR, Lemma 69.20.2]{stacks-project}. Let $U$ denote the affine scheme $\Spec(\Sym_k(\tilde{U}^\vee))$.



Let $Z=\{(\sigma, p)\in U\times X\colon \sigma(p)=0\}$ be defined as the derived pullback in the following diagram, note all squares are homotopy cartesian:
\begin{equation} \label{eq:univ.-0locus-orientation}
  \begin{tikzcd}
        Z(\sigma) \arrow{r} \arrow{d} 
        & Z\arrow[bend right =45,swap,"f"{xshift=6pt,yshift=22pt}]{dd}  \arrow["i_Z"]{d}\arrow{r} & X\arrow{d}{0}  \\
         X \arrow{d}\arrow{r} & U\times X \arrow{d}{p_U} \arrow{r}{ev} & V \\
        \Spec k \arrow{r}{\sigma} & U
    \end{tikzcd}  
\end{equation}
The map $\Spec k\to U$ picks out the $k$-point of $U$ corresponding to $\sigma$. The map $ev\colon U\times X\to V$ is the evaluation map $(\sigma, p)\mapsto \sigma(p)$, and the right vertical map is the zero section. 
Denote by $f\colon Z\to U$ the composition $Z\to U\times X$ with the projection $p_U\colon U\times X\to U$. 
The morphism $i_Z$ is the derived base change of the zero
section $X\to V$, hence is finite unramified and quasi-smooth.
Since $X\to \Spec k$ is equivariantly smoothable,
$p_U\colon U\times X\to U$ is equivariantly smoothable, and hence so
is $f$. Moreover, $f$ is proper and quasi-smooth.
\end{construction} 

\begin{remark}Note that $Z$ is the universal zero locus for sections contained in $U$, i.e., for any section $s \in\tilde{U}$, we have a homotopy cartesian square:
\begin{equation} \label{eq:0locus-univ.pullback}
    \begin{tikzcd}
        Z(s) \arrow{r} \arrow{d} & Z \arrow{d}{f}  \\
        \Spec(k)  \arrow{r}{s} & U,
    \end{tikzcd}
\end{equation}
where $Z(s)$ is the derived zero locus of the section $s\colon X\to V$. 
\end{remark}

The map $Z(\sigma) \to \Spec(k)$ has an $E_G$-orientation induced by pulling back the relative orientation of $f$ as follows. 
Suppose $V \to X$ is $E_G$-oriented for some $E_G \in \SH^G(k)$ motivic ring  spectrum $E_G$ in $\SH^G(S)$, $S=\Spec (k)$. Then there will be an induced orientation on $f\colon Z \to U$. Indeed, 
we can pull $V$ back to $U\times X$ along the projection $p_X\colon U\times X\to X$ to obtain a relatively oriented vector bundle $p_X^*V\to U\times X$ over $U$. The evaluation map defines a section of $p_X^*V\to U\times X$, and $i_Z: Z \to U \times X$ is the derived zero locus of this section. 
Therefore $L_{i_Z} \simeq i_Z^* p_X^* \calE[1]$.
Since $f = p_U \circ i_Z$, this equivalence gives a fiber sequence of cotangent complexes:
\[
i_Z^* p_X^* L_X \to L_f \to i_Z^* p_X^* \calE[1].
\]
The relative orientation of $f$ is given by
\begin{equation}\label{eq:rel_or_of_f}
E_G (Z,L_f) \simeq E_G(Z, i_Z^* p_X^* (L_X - \calE)) \simeq E_G(Z),
\end{equation}
where the last equivalence is given by the pullback to $Z$ of 
the orientation \eqref{eq:VB orientation} of the vector bundle $V \to X$.

The following theorem shows that the equivariant Euler number of $V$ with respect to a section $\sigma$ can be computed in terms of the equivariant local degree of the map $f$ at the rational point defined by $\sigma$ in $U$.

\begin{theorem}\label{thm:non_transverse_G_euler=local_degrees}
    Let $G$ be a finite group that is tame over $k$, and let $E_G\in \SH^G(S)$ be a motivic ring spectrum, $S=\Spec(k)$. Let $p\colon V = \mathbf V(\mathcal E) \to X$ be an $E_G$-oriented vector bundle over a proper, quasi-smooth, equivariantly smoothable derived
    $G$-algebraic space $X$ over $\Spec k$. Assume that $L_{X/S}-\mathcal E$
    has constant virtual rank $0$.
    Let $\sigma$ be a $G$-equivariant section such that the classical truncation $\pi_0Z(\sigma)$ is zero-dimensional.
    Let $f$ be as in Construction \ref{construction:map_f_for_euler} with the relative orientation as in \eqref{eq:rel_or_of_f}. Then we have an equality in $E^0_G(S)$,
    \[
    n^{E_G}(V,\sigma) = \sum_{\{G\cdot x\subset |\pi_0Z(\sigma)|\}}
    \Tr_{G_x}^G\deg_x^{G_x}(f).
    \] 
\end{theorem}
\begin{proof}
 Note that the $E_G$-orientation of the zero locus $Z(\sigma)$ (over $\Spec(k)$) of the section $\sigma$ defined in \eqref{eq:0locus-orientation}  coincides with the pull back of the relative orientation of $f$ defined in \eqref{eq:univ.-0locus-orientation} under the pullback diagram \eqref{eq:0locus-univ.pullback}. 

By oriented base change (Proposition \ref{prop:oriented_base_change}) applied to the square  \eqref{eq:0locus-univ.pullback} and the definitions of degree and Euler class (Definitions \ref{defn:global_degree} and \ref{defn:Euler-class} respectively) we get
\[
n^{E_G}(V, \sigma) = \sigma^*(\deg^G(f)).
\]
Now by local to global degree Theorem~\ref{thm:local_global_for_schemes}, we have the required equality.
\end{proof}

We obtain the following version of Bézout's theorem for the (possibly non-transverse) intersection of the derived zero locus of sections of equivariant oriented vector bundles over an algebraic space $X$ with $G$-action.

   \begin{theorem}[Equivariant non-transverse B\'ezout's theorem]\label{thm:bezout} Let $k$ be a field, $G$ be a finite group that is tame over $k$, $S=\Spec(k)$, and  $E_G \in \SH^G(S)$ be an equivariant motivic ring spectrum. Let $X$ be a proper, smooth algebraic space over $S$ with $G$-action. 
   Let $\calE_1, \ldots, \calE_n$ be $G$-equivariant locally free sheaves on $X$ over $S$ such that $\sum_{i=1}^n \op{rank}(\calE_i)=\dim X$ and $\mathbb{V}(\oplus_{i=1}^n\calE_i)$ defines an $E_G$-oriented vector bundle on $X$ relative to $S$. 
   Let $\sigma_1, \ldots \sigma_n$ be $G$-invariant sections of $\calE_1, \ldots, \calE_n$, respectively.  Let $Z_1, \ldots, Z_n$ denote the derived vanishing loci of the sections $\sigma_1, \ldots, \sigma_n$ respectively. 
   
   The $E_G$-intersection number of $Z_1, \ldots, Z_n$ is given by $n^{E_G}(\mathbb{V}(\oplus_{i=1}^n\calE_i))$, and $n^{E_G}(\mathbb{V}(\oplus_{i=1}^n\calE_i))$ is independent of choice of $G$-invariant sections $\sigma_1, \ldots, \sigma_n$. If $\sigma$ is a section of $\mathbb{V}(\oplus_{i=1}^n\calE_i)$  with isolated, but not necessarily simple, zeros $Z(\sigma)=\{x_1, \cdots, x_k\}$, then the $E_G$-intersection number can be computed as
   \[
   \sum_{\{G\cdot x_i\subseteq Z(\sigma)\}} \Tr^G_{G_{x_i}} \deg_{x_i}^{G_{x_i}}f
   \]
   in $E^0_G(S)$, where $f$ is as in Construction \ref{construction:map_f_for_euler} and $G_{x}$ denotes the $G$-stabilizer at a point $x \in X$. When $Z(\sigma)$ consists of zeros that are both isolated and simple and $\sigma$ is compatibly $E_G$-oriented, the $E_G$-intersection number can be computed as
   \[
   \sum_{\{G\cdot x_i\subseteq Z(\sigma)\}} \Tr^G_{G_{x_i}} \Tr_{k(x_i)/k} (1_{G_{x_i}})
   \]
   in $E^0_G(S)$.  
   \end{theorem}

   \begin{proof}
       In view of \eqref{eq:int-no.=Euler-no.}, this follows from Theorem \ref{thm:non_transverse_G_euler=local_degrees} in the isolated case and Proposition~\ref{prop:Euler-number=local_deg} in the isolated and simple case by induction on the number of $G$-equivariant vector bundles. 
   \end{proof}

\begin{remark}\label{rmk:classical_transversality}
 The computation of $\sum_{\{G\cdot x_i\subseteq Z(\sigma)\}} \Tr^G_{G_{x_i}} \deg_{x_i}^{G_{x_i}}f$ in Theorem \ref{thm:bezout} does not require the section $\sigma\colon X\to V$ to have simple \emph{and} isolated zeros, the section is only required to have isolated zeros. Under the hypotheses of Theorem \ref{thm:bezout}, the derived zero locus $Z(\sigma)$ is in fact classical, though it may be nonreduced when the zeros are not simple. 
The derived formulation is nevertheless useful because $Z(\sigma)$ occurs as the fiber of the universal quasi-smooth morphism in \eqref{eq:0locus-univ.pullback}, so derived base change applies without a transversality assumption. This allows us to prove the local multiplicity statements like Theorems \ref{thm:bezout} and \ref{thm:non_transverse_G_euler=local_degrees} uniformly for simple and non-simple isolated zeros.
\end{remark}

When a section $\sigma$ does intersect the zero section transversally, e.g., when the intersection of $Z_1,\ldots, Z_n$ is transverse in the case of Theorem \ref{thm:bezout}, the zeros of $\sigma$ are isolated and simple, and $Z(\sigma)$ is finite étale over $k$.

\begin{exa}\label{example:Z4}
    Let $k$ be a field with $\op{char}(k)\neq 2$. Let $G=C_4= \langle g\rangle$ act on $\mathbb{P}_k^2$ by $g\cdot [x:y:z] = [-y:x:z]$.\footnote{We will write $G$ when using $C_4$ creates notational clutter.} Let 
    \[
    C_1 = V(x^2+y^2-z^2),\qquad C_2 = V(xy).
    \]
    The conics $C_1$ and $C_2$ determine an equivariant section of $\calO(2)\oplus \calO(2)\to \mathbb{P}_k^2$ given by $(x^2+y^2-z^2, xy)$. 
    The zeros of the section are exactly the points of intersection of $C_1$ and $C_2$, which as a set are
    \[
    \{C_1\cap C_2\} = \{[1:0:1], [0:1:1], [-1:0:1], [0:-1:1]\}. 
    \]
    The intersection $C_1\times_{\mathbb{P}^2_k} C_2$ is a $C_4$-invariant subscheme of $\mathbb{P}_k^2$. Further, the four intersection points lie in a single $C_4$-orbit and all are defined over $k$. By Theorem \ref{thm:bezout}, the $C_4$-equivariant intersection in $\KGL$ is given by the equivariant Euler number of $V:=\mathbb{V}(\calO(2)\oplus \calO(2))$,  
    \begin{align*}
    n^{\KGL_{C_4}}(V) &=  \sum_{G\cdot x_i\subseteq Z(\sigma)} \Tr^G_{G_{x_i}} \Tr_{k(x_i)/k} (1_{G_{x_i}}) \\
    &= \Tr^G_{G_{x_i}} (1_{G_{x_i}}) \\ 
    &= \Tr^{C_4}_{e} (\trivrep_k) \\ 
    &= \ind_e^{C_4} \trivrep_k \\ 
    &= k[C_4]
    \end{align*}
    in the representation ring $R_k(C_4)\cong \KGL^0_{C_4}(\Spec k)$, where $\trivrep_k$ is the trivial representation over $k$. 
\end{exa}

\begin{exa}\label{example:A4}
    We work over $\mathbb{R}$ and let $E=\KGL$. Let $G=A_4$ act on $\mathbb{P}_{\mathbb{R}}^2$ by restricting the standard $S_4$-representation on $\mathbb{R}^3$ to $A_4$ and taking the projectivization. Let
    \[
    C_1 = V(x^2-y^2+z^2-2xz), \hspace{2mm} C_2 =V(x^2-\frac{1}{2}y^2+z^2-xy-yz). 
    \]
    The conics $C_1$ and $C_2$ determine the equivariant section of $\calO(2)\oplus \calO(2)\to \mathbb{P}_{\mathbb{R}}^2$ given by $(x^2-y^2+z^2-2xz, x^2-\frac{1}{2}y^2+z^2-xy-yz)$. The zeros of the section are exactly the  intersection $C_1\times_{\mathbb{P}_{\mathbb{R}}^2} C_2$, which as a set is
    \[
    \{C_1\cap C_2\} = \{[1:2:3], [1:2:-1], [1:-2:-1], [-3:-2:-1]\}. 
    \]
    This is an $A_4$-invariant subscheme of $\mathbb{P}_{\mathbb{R}}^2$. Further, the four intersection points lie in a single $A_4$-orbit with stabilizer $A_3$. By Theorem \ref{thm:bezout}, as in Example \ref{example:Z4} the equivariant intersection in $\KGL$ is given by the Euler number of $V:=\mathbb{V}(\calO(2) \oplus \calO(2))$ 
    \begin{align*}
    n^{KGL_{A_4}}(V) &=  \sum_{G\cdot x_i\subseteq Z(\sigma)} \Tr^G_{G_{x_i}} \Tr_{k(x_i)/k} (1_{G_{x_i}}) \\
    &= \Tr^{A_4}_{A_3} (\trivrep_\mathbb{R}) \\ 
    &= \ind_{A_3}^{A_4} \trivrep_\mathbb{R} \\ 
    &= \mathbb{R}[A_4/A_3] \\ 
    & = \trivrep_{\mathbb{R}}+W
    \end{align*}
    in the representation ring $R_{\mathbb{R}}(A_4)\cong \KGL^0_{A_4}(\Spec\mathbb{R})$, where $\trivrep_\mathbb{R}$ is the trivial representation over $\mathbb{R}$ and $W$ is the unique 3-dimensional irreducible representation of $A_4$. 
\end{exa}

\begin{exa}\label{example:Z3}
We will work over $\mathbb{C}$ in this example. As before, let $E=\KGL$. Let $G=C_3 = \langle g\rangle$ act on $\mathbb{P}^2_{\mathbb{C}}$ by $g\cdot[x:y:z] = [z:x:y]$. Let 
\[
C_1 = V(x^2+y^2+z^2), \hspace{2mm}  C_2 = V(xy+xz+yz).
\] The conics $C_1$ and $C_2$ determine an equivariant section of $\calO(2)\oplus \calO(2)\to \mathbb{P}^2_{\mathbb{C}}$ given by $(x^2+y^2+z^2, xy+xz+yz)$. The zeros of the section  are exactly the points of intersection, which as a set are
    \[
    \{C_1\cap C_2\} = \{[1:\omega:\omega^2], [1:\omega^2:\omega]\},
    \]
where $\omega$ is a cube root of unity and each intersection point occurs with multiplicity 2. The intersection $C_1\times_{\mathbb{P}^2_{\mathbb{C}}} C_2$ is a $C_3$-invariant subscheme of $\mathbb{P}^2_{\mathbb{C}}$ with each point (projectively) fixed by the group action. Writing $p_1=[1:\omega:\omega^2]$ and $p_2=[1:\omega^2:\omega]$, by Theorem \ref{thm:non_transverse_G_euler=local_degrees} the Euler number of $V:= \mathbb{V}(\calO(2)\oplus \calO(2))$ is
 \[
 n^{\KGL_{C_3}}(V) = \deg^{C_3}_{p_1}f + \deg^{C_3}_{p_2}f,
\]
 where $f$ is as in Theorem \ref{thm:non_transverse_G_euler=local_degrees} and we consider the points $p_1$ and $p_2$ as derived points. Writing $Z:=C_1 \times_{\mathbb{P}^2_{\mathbb{C}}}^\mathbf{R} C_2$ for the derived intersection  and $i\colon Z\to \Spec(\mathbb{C})$, we equivalently compute $$\deg^{C_3} Z = i_!(1_Z) \in \KGL^0_{C_3}(\Spec\mathbb{C}) \cong R_{\mathbb{C}}(C_3).$$

$R^ji_*\mathcal{O}_Z = 0$ for $j>0$ since $i$ is finite, whence $i_!(1_Z) = [i_*\calO_Z] = \Gamma(Z,\calO_Z)$. 
Furthermore, 
\[
\deg^{C_3} Z = [i_*\calO_Z] = \Gamma(Z,\calO_Z) \cong \mathbb{C}[s]/(s^2)\oplus \mathbb{C}[t]/(t^2)
\]
in $R_{\mathbb{C}}(C_3)$, where the $C_3$ action on $\mathbb{C}[s]/(s^2)\oplus \mathbb{C}[t]/(t^2)$ is given by $g\cdot (s,t) = (\omega s, \omega^2 t)$ and trivial action on $\mathbb{C}$ for $g$ a generator of $C_3$. 

We wish to write $\deg^{C_3} Z$ as a sum of irreducible representations. Denote by $\chi$ the 1-dimensional irreducible $C_3$ representation given by multiplication by $\omega$, and denote by $\chi^2$ the 1-dimensional irreducible representation given by multiplication by $\omega^2$. It is clear that each of $\mathbb{C}[s]/(s^2)$ and $\mathbb{C}[t]/(t^2)$ contains a trivial representation, $\mathbb{C}[s]/(s^2)$ contains $\chi$, and $\mathbb{C}[t]/(t^2)$ contains $\chi^2$. Thus
\begin{align*}
    n^{\KGL_{C_3}}(\mathbb{V}(\calO(2)\oplus\calO(2))) &= \deg^{C_3} Z \\ 
    &= \Gamma(Z,\calO_Z) \\ 
    &= \mathbb{C}[s]/(s^2)\oplus \mathbb{C}[t]/(t^2) \\
    &= \trivrep_{\mathbb{C}} + \chi + \trivrep_{\mathbb{C}} + \chi^2 \\ 
    &=\trivrep_{\mathbb{C}} + \rho
\end{align*}
in $R_{\mathbb{C}}(C_3)$, where $\trivrep_{\mathbb{C}}$ is the trivial representation over $\mathbb{C}$ and $\rho$ is the regular representation of $C_3$. 
\end{exa}

\begin{exa}\label{example:Z3_non_lci}
    Consider the $G=C_3$ action on $\mathbb{P}^2_\mathbb{C}$ from Example \ref{example:Z3}, and again consider 
    \[
    C = V(x^2+y^2-z^2) \cong \mathbb{P}^1_{\mathbb{C}}. 
    \]
Write $f = x^2+y^2-z^2$. We will consider the derived intersection of $C$ with itself, 
\[
Z:=C \times_{\mathbb{P}^2_{\mathbb{C}}}^\mathbf{R} C
\]
and compute $\deg^{C_3}Z$ in $\KGL^0_{C_3}(\Spec \mathbb{C}) \cong R_{\mathbb{C}}(C_3)$. In this case, the intersection is not lci, and $\pi_0(Z) = C$ is not equivalent to $Z$. Thus we must compute $\deg^{C_3}Z$ as a derived intersection rather than a scheme-theoretic intersection as in Example \ref{example:Z3}. By section independence, this calculation should still give $\trivrep + \rho$ as in Example \ref{example:Z3}. 

First we compute 
\[
\calO_C \times_{\calO_{\mathbb{P}^2_{\mathbb{C}}}}^\mathbf{L} \calO_C. 
\]
Consider the Koszul resolution of $\calO_C$, 
\[
0\to \calO_{\mathbb{P}^2_{\mathbb{C}}}(-2) \stackrel{\cdot f}{\to} \calO_{\mathbb{P}^2_{\mathbb{C}}} \to \calO_C \to 0,
\]
where $\calO_{\mathbb{P}^2_{\mathbb{C}}}(-2) \stackrel{\cdot f}{\to} \calO_{\mathbb{P}^2_{\mathbb{C}}} $ is multiplication by $f$. Tensoring the resolution with $\calO_C$, we obtain 
\[
\left[\calO_{\mathbb{P}^2_{\mathbb{C}}}(-2) \stackrel{\cdot f}{\to} \calO_{\mathbb{P}^2_{\mathbb{C}}} \right]\otimes_{\calO_{\mathbb{P}^2_{\mathbb{C}}}} \calO_C = \calO_C(-2)\stackrel{\cdot 0}{\to} \calO_C, 
\]
whence 
\[ 
\calO_C \otimes_{\calO_{\mathbb{P}^2_{\mathbb{C}}}}^\mathbf{L} \calO_C = \calO_C\oplus \calO_C(-2)[1]
\]
in $D^b(\mathrm{Coh}(C))$. 

We consider $\deg^{C_3}Z$ as the exceptional pushforward of the class $[\calO_C]-[\calO_C(-2)]$ in $\KGL^0_{C_3}(C)$ along $\pi_C\colon C\to \Spec \mathbb{C}$. 
Note since $C\cong \mathbb{P}^1_{\mathbb{C}}$ and $\deg \calO_C(1) =2$, $\calO_C\cong \calO_{\mathbb{P}^1_{\mathbb{C}}}$ and $\calO_C(-2)\cong \calO_{\mathbb{P}^1_{\mathbb{C}}}(-4)$. Note also that this isomorphim is equivariant if we take the $C_3$ action on $\mathbb{P}^1_{\mathbb{C}}$ to be the action on $C$ composed with the isomorphism $C\cong \mathbb{P}^1_{\mathbb{C}}$. Specifically, the action on $\mathbb{P}^1_{\mathbb{C}}$ is $[s,t] \mapsto [s+t,t-s]$. 
Thus 
\[
\calO_C \otimes_{\calO_{\mathbb{P}^2_{\mathbb{C}}}}^\mathbf{L} \calO = \calO_{\mathbb{P}^1_{\mathbb{C}}}\oplus \calO_{\mathbb{P}^1_{\mathbb{C}}}(-4)[1].
\]

The oriented pushforward of this class computes $\deg^{C_3}Z$, and is 
\begin{align*}
(\pi_{C})^{\mathrm{or}}_!\left([\calO_{\mathbb{P}^1_{\mathbb{C}}}] - [\calO_{\mathbb{P}^1_{\mathbb{C}}}(-4)]\right) & = \chi^{C_3}(\calO_{\mathbb{P}^1_{\mathbb{C}}}) - \chi^{C_3}(\calO_{\mathbb{P}^1_{\mathbb{C}}}(-4)) \\
& = \trivrep_{\mathbb{C}} - \chi^{C_3}(\calO_{\mathbb{P}^1_{\mathbb{C}}}(-4)) \\
& = \trivrep_{\mathbb{C}} - \sum_i(-1)^i[H^i(\mathbb{P}^1_{\mathbb{C}}, \calO_{\mathbb{P}^1_{\mathbb{C}}}(-4))]. 
\end{align*}
Since $H^0(\mathbb{P}^1_{\mathbb{C}}, \calO_{\mathbb{P}^1_{\mathbb{C}}}(-4)) \simeq 0$ and by Serre duality, $H^1(\mathbb{P}^1_{\mathbb{C}}, \calO_{\mathbb{P}^1_{\mathbb{C}}}(-4))\cong H^0(\mathbb{P}^1_{\mathbb{C}}, \calO_{\mathbb{P}^1_{\mathbb{C}}}(2))^{\vee}$, whence 
\[
\deg^{C_3} Z = \pi_{C,!}\left([\calO_{\mathbb{P}^1_{\mathbb{C}}}] - [\calO_{\mathbb{P}^1_{\mathbb{C}}}(-4)]\right)  = \trivrep_{\mathbb{C}} +[H^0(\mathbb{P}^1_{\mathbb{C}}, \calO_{\mathbb{P}^1_{\mathbb{C}}}(2))^{\vee}]. 
\]
Writing $V$ for the 2-dimensional $C_3$ representation such that $\mathbb{P}^1_{\mathbb{C}} \cong \mathbb{P}V$, it is well known that $H^0(\mathbb{P}^1_{\mathbb{C}}, \calO_{\mathbb{P}^1_{\mathbb{C}}}(2))^{\vee} = \op{Sym}^2V^\vee$. To actually compute this, note that restriction along $C\hookrightarrow \mathbb{P}^2_{\mathbb{C}}$ gives $H^0(\mathbb{P}^2_{\mathbb{C}}, \calO(1))\cong H^0(\mathbb{P}^1_{\mathbb{C}}, \calO(2))$, which is the 3-dimensional permutation representation $\trivrep_{\mathbb{C}} + \chi +\chi^2$, using the notation from Example \ref{example:Z3}. Thus 
\[
[H^0(\mathbb{P}^1_{\mathbb{C}}, \calO_{\mathbb{P}^1_{\mathbb{C}}}(2))^{\vee}] = (\trivrep_{\mathbb{C}} +\chi +\chi^2)^\vee = \trivrep_{\mathbb{C}} +\chi^2 + \chi.
\]
We conclude that 
\[
\deg^{C_3} Z = \trivrep_{\mathbb{C}} + (\trivrep_{\mathbb{C}} + \chi+\chi^2) = \trivrep_{\mathbb{C}} + \rho, 
\]
as expected by Example \ref{example:Z3} and section independence (Proposition \ref{prop:indep.-of-sec}). 
\end{exa}

 In the examples above, we have not shown that $\sigma$ is compatibly oriented or that the bundle $\mathbb{V}(\calO(2)\oplus \calO(2))$ are $\KGL$-oriented. This is because all bundles are canonically $\KGL$-oriented under our assumptions by \cite[Theorem~1.3(3)]{hoyois-cdh}, and these orientations are natural with respect to pullbacks by arbitrary morphisms. Note that by Proposition \ref{prop:indep.-of-sec}, in each of the examples above, any other choice of sections defining conics that are invariant under the same group action will have equivariant intersection number equal to those computed above.

\bibliographystyle{amsalpha}
\bibliography{bib_alg_G_degree}

{\small
\noindent
Department of Mathematics, Brown University, Providence, Rhode Island, 02912, USA 

\noindent
School of Mathematics, Tata Institute of Fundamental Research, Mumbai,
400005, India}

\end{document}